\documentclass{amsart} 
\usepackage{mathrsfs}
\usepackage{amssymb,amsmath,amsthm,amsfonts}
\usepackage[colorlinks=true,linkcolor=blue,citecolor=blue]{hyperref}
\usepackage[inline]{enumitem}  
\newcommand{\R}{\mathbb R}

\newcommand{\Z}{\mathbb Z}
\newcommand{\T}{\mathbb T}
\newcommand{\N}{\mathbb N}

\newcommand{\F}{ {\mathcal F} }

\newcommand{\eps}{\varepsilon}
\newcommand{\si}{\sigma}

\DeclareMathOperator{\RE}{Re}

\def\Eq#1#2{\mathop{\sim}\limits_{#1\rightarrow#2}}
\def\Tend#1#2{\mathop{\longrightarrow}\limits_{#1\rightarrow#2}}

\def\TagOnRight

\theoremstyle{plain}
\newtheorem{theorem}{Theorem} [section]
\newtheorem{lemma}[theorem]{Lemma}
\newtheorem{corollary}[theorem]{Corollary}
\newtheorem{proposition}[theorem]{Proposition}

\theoremstyle{remark}
\newtheorem{remark}[theorem]{Remark}

\theoremstyle{definition}
\newtheorem{definition}[theorem]{Definition}

\newtheorem{assumption}[theorem]{Assumption}

\def\d{{\partial}}

\def\Eq#1#2{\mathop{\sim}\limits_{#1\rightarrow#2}}
\def\Tend#1#2{\mathop{\longrightarrow}\limits_{#1\rightarrow#2}}

\def\({\left(}
\def\){\right)}
\def\<{\left\langle}
\def\>{\right\rangle}

\def\Sch{{\mathcal S}}% Schwartz space

\numberwithin{equation}{section}

\begin{document}
\title[Norm inflation  for NLS  in Fourier-Lebesgue and modulation spaces]
{Norm inflation for nonlinear Schr\"odinger equations   in
 Fourier-Lebesgue and modulation spaces of negative regularity} 
\author[D. Bhimani]{Divyang G. Bhimani}
\author[R. Carles]{R\'emi Carles}
\address{Univ Rennes, CNRS\\ IRMAR - UMR 6625\\ F-35000
  Rennes, France}
\email{divyang@tifrbng.res.in}
\email{Remi.Carles@math.cnrs.fr}

\thanks{The first author benefits from a fellowship by the Henri
  Lebesgue Center, which he wishes to thank. The first author also wishes to thank DST-INSPIRE and TIFR-CAM  for the academic leave. The second author is
  supported by Rennes M\'etropole through its AIS program.}


\maketitle
\begin{abstract}
  We consider nonlinear Schr\"odinger
  equations in Fourier-Lebesgue and modulation spaces involving
  negative regularity. The equations
  are posed on the whole space, and 
  involve a smooth power 
  nonlinearity.
  We prove two types of 
  norm inflation results. We first establish  norm inflation results below the
  expected critical regularities. We then prove norm inflation with
  infinite loss of regularity under less general assumptions. To do so, we recast
  the theory of multiphase weakly nonlinear geometric optics for
  nonlinear Schr\"odinger equations in a general abstract functional
  setting. 
\end{abstract}
%\tableofcontents

\section{Introduction}

\subsection{General setting}
\label{sec:general-setting}

We consider the  nonlinear Schr\"odinger (NLS) equations of the form 
\begin{equation}\label{nls}
i\partial_t \psi + \frac{1}{2} \Delta \psi = \mu |\psi|^{2\sigma} \psi,
  \quad x\in \R^d;\quad  \psi(0, x)=\psi_0(x),
\end{equation}
where $\psi=\psi(t,x)\in \mathbb C$, $\sigma \in \mathbb N$, $\mu\in
\{ 1, -1\}.$ We prove some ill-posedness results in Fourier-Lebesgue
and modulation spaces, involving negative regularity in space. We recall the notion of well-posedness in the sense of Hadamard.
\begin{definition} Let $X, Y \hookrightarrow \mathcal{S}'(\R^d)$ be a Banach spaces.  The Cauchy problem for \eqref{nls} is well posed from  $X$ to $Y$ if, for all bounded subsets $B\subset X$, there exist $T>0$ and a Banach space  $X_{T} \hookrightarrow C([0,T], Y)$ such that:
\begin{itemize}
\item[(i)] For all $\phi \in X$,  \eqref{nls}  has a unique solution $\psi\in X_T$ with  $\psi_{|t=0}=\phi$. 
\item[(ii)]  The mapping $\phi\mapsto \psi$ is continuous from $(B, \|\cdot\|_{X})$ to $C([0,T],Y).$
\end{itemize}
\end{definition}
The negation of the above definition is called a \textbf{lack of
  well-posedness} or \textbf{instability}. In connection with the
study of ill-posedness of \eqref{nls} and nonlinear wave
equations  Christ, Colliander, and Tao introduced in
\cite{christ2003ill}  the
notion of \textbf{norm inflation} with respect to a given (Sobolev)
norm, saying that there exist a sequence of smooth initial data
$(\psi_n(0))_{n\geq 1}$ and a sequence of times $(t_n)_{n\geq 1}$,
both converging to $0$, so that the corresponding smooth solution
$\psi_n$, evaluated at $t_n$, is unbounded (in the same space).
\smallbreak 

The solutions  to \eqref{nls} is invariant  under  the scaling transformation
\begin{equation}
  \label{eq:scaling}
  \psi(t,x) \mapsto \lambda^{1/\sigma} \psi \(\lambda^2 t, \lambda x\), \quad \lambda>0. 
\end{equation}
The homogeneous Sobolev space $\dot H^{s}(\R^d)$ is invariant exactly for
$s=s_c$, where
\[ s_c= \frac{d}{2} -\frac{1}{\sigma}.\]
Another important invariance of \eqref{nls} is the
Galilean invariance: if $\psi(t,x)$ solves \eqref{nls}, then so does
\begin{equation}
  \label{eq:galilean}
  e^{iv\cdot x -i |v|^2 t/2}\psi(t,x-vt)
\end{equation}
for any $v \in \R^d$. This transform does not alter the $L^2(\R^d)$
norm of the function. From these two invariances, well-posedness is not
expected to hold  in  $H^s(\R^d)$ as soon as $s<\max(0,s_c)$. In this paper,
we consider the case of negative regularity, $s<0$, in
Fourier-Lebesgue and modulation spaces, instead of Sobolev spaces.
\smallbreak

Kenig, Ponce and Vega \cite{KPV01}  established instability
for the cubic NLS in $H^{s}(\R)$ for $s<0$.   Christ, Colliander
and Tao \cite{christ2003ill} generalized this result in $H^{s}(\R^d)$
for $s<0$ and $d\ge 1$.  In the periodic case $x\in \T^d$, instability  in
$H^{s}(\T^d)$ for $s<0$ was
established in \cite{CCTper} ($d=1$) and \cite{carles2010multiphase}
($d\ge 1$). Stronger results for the cubic NLS on the circle were
proven by Molinet \cite{Molinet}. In \cite[Theorem 1.1]{oh2017remark},
Oh established  norm-inflation 
for  \eqref{nls} in the cubic case $\sigma=1$, in $H^{s}(\mathbb
R)$ for $s\leq  -1/2$ and  in $H^s(\R^d)$ for $s<0$ if $d\geq 2.$ He
actually proved that the flow map fails to be continuous at any
function in $H^s$, for $s$ as above. Norm inflation in the case of
mixed geometries, $x\in \R^d\times \T^n$, for sharp negative Sobolev
regularity in \eqref{nls},  is due to Kishimoto
\cite{kishimoto2019}, who also considers nonlinearities which are not
gauge invariant. 
\smallbreak

The general picture to prove ill-posedness results is typically as
following, as explained in e.g. \cite{christ2003ill}:  at negative
regularity, one relies on a transfer from 
high frequencies to low frequencies, while to prove ill-posedness at
positive regularity, one uses a transfer from
low frequencies to high frequencies. In particular, the proofs are
different whether a negative or a positive regularity is considered.
\smallbreak

Stronger phenomena than norm inflation have also been proved, showing
that the flow map fails to be continuous at the origin from $H^s$ to
$H^k$ even for (some) $k<s$, and so a loss of regularity is present.
This was proven initially  for $0<s<s_c$ by Lebeau
\cite{Lebeau05} in
the case of the wave equation, then in \cite{CaARMA} (cubic nonlinearity)
and \cite{ACMA,ThomannAnalytic} for NLS. In the case of negative
regularity, an infinite loss of regularity was established in
\cite{carles2012geometric} for \eqref{nls} in $H^s(\R^d)$ ($d\ge 2$
and $s<-1/(2\sigma+1)$), and in the periodic case $x\in \T^d$ in
   \cite{carles2017norm}, in  Fourier-Lebesgue  spaces. Typically,
   the NLS flow map fails to be continuous at the origin from 
$H^s(\R^d)$ to $H^k(\R^d)$, for any $k\in \R$. 

\subsection{Fourier-Lebesgue spaces}

The Fourier-Lebesgue space $\mathcal{F}L_s^p(\mathbb R^d)$ is defined by 
\[\mathcal{F}L_s^p(\mathbb R^d)= \left\{f\in \mathcal{S}'(\mathbb
  R^d): \|f\|_{\mathcal{F}L_s^{p}}:= \| \hat{f} \langle \cdot
  \rangle^s  \|_{L^{p}}< \infty \right\},\]
where the Fourier transform is defined as
\begin{equation*}
  \hat f(\xi)= \F f(\xi) = \frac{1}{(2\pi)^d}\int_{\R^d}e^{-ix\cdot
    \xi}f(x)dx,\quad f\in \Sch(\R^d),
\end{equation*}
and where $1\leq p \leq \infty$, $s\in \R$, and $\langle \xi
\rangle^{s} = (1+ |\xi|^2)^{s/2}$  ($\xi \in \mathbb R^d$). For $p=2$,
$\mathcal{F}L_s^2=H^s$ the usual Sobolev space.
For $s=0$,
we write $\mathcal{F}L^p_0(\R^d) = \mathcal{F}L^p(\R^d).$ 

\smallbreak

The scaling \eqref{eq:scaling} leaves the \emph{homogeneous}
$\mathcal{F}\dot L_s^p(\mathbb
R^d)$-norm (replace the Japanese bracket $\<\cdot\>^s$ with the length
$|\cdot|^s$ in the definition of  $\mathcal{F}L_s^p(\mathbb
R^d)$) invariant for $s=s_c(p)$, where
\[ s_c(p):= d \left( 1-\frac{1}{p} \right) -\frac{1}{\sigma}.\]
Of course when $p=2$, we recover the previous value $s_c$. On the
other hand, the Galilean transform \eqref{eq:galilean} does not alter
the $\F L^p(\R^d)$ norm of $\psi$, and so well-posedness is not
expected to hold  in  $\F L^p_s$ for $s<\max (0,s_c(p))$. Note however
that the recent results from \cite{HGKV-p} show that this heuristical argument
is not always correct: in the case $p=2$, $d=1=\sigma$, well-posedness may hold for
$\min  (0,s_c(2))=-\tfrac{1}{2}<s<\max (0,s_c(2))=0$. Therefore,
$s<\min (0,s_c(p))$ is 
a safer assumption to obtain ill-posedness results. In this
paper, we consider cases where $s<0$.
\smallbreak

In \cite[Theorem 1]{hyakuna2012existence}, Hyakuna-Tsutsumi
established local well-posedness for  the cubic  NLS  in
$\mathcal{F}L^p(\R^d)$  for $p \in (4/3, 4)\setminus 
\{2\}.$ Later this result is generalized in  \cite[Theorem
1.1]{hyakuna2018multilinear}  for 
$p\in [1,2].$  
\smallbreak

Our first results concern norm inflation of the type discussed above:
\begin{theorem}\label{theo:inflationFLp}
  Assume that   $1\leq p \leq \infty, d, \sigma \in \mathbb N$ and
  $s< \min \(0, s_c(p)\)$.  For any $\delta>0,$ there exists $\psi_0 \in \mathcal{F}L^p_s(\R^d)$ and $T>0$ satisfying  
\begin{equation*}
\| \psi_0 \|_{\mathcal{F}L^p_s} < \delta \quad \text{and} \quad   0<T< \delta,
\end{equation*}
such that the corresponding solution $\psi$  to \eqref{nls} exists on  $[0, T]$ and 
\[\|\psi(T)\|_{\mathcal{F}L^p_s}> \delta^{-1}. \]
\end{theorem}
As discussed above, in the case $s_c(p)>0$, norm inflation is expected
in $ \mathcal{F}L^p_s(\R^d)$ for $0<s<s_c(p)$, but with different
arguments. The proof of Theorem~\ref{theo:inflationFLp} is inspired by
the two-scale analysis of Kishimoto \cite{kishimoto2019}. We also
prove norm inflation with an infinite loss of regularity: the initial 
 regularity must be sufficiently small, and we leave out the cubic
 one-dimensional nonlinearity.

\begin{theorem}\label{theo:fourier-lebesgue}
 Let $\sigma\in \N$,  $s< -\frac{1}{2 \sigma
    +1}$ and assume $d\sigma \geq 2.$   There exist a  sequence of
  initial data  $\(\psi_n(0)\)_{n\geq 1}$ in  $ \Sch(\R^d)$ such that 
\[ \|\psi_n(0)\|_{ \mathcal{F}L^p_{s}} \Tend n \infty 0, \quad \forall
  p\in [1,\infty],\]
and a sequence of times  $t_n \to 0$ such that the corresponding solutions $\psi_n$ to \eqref{nls} satisfies
\[ \|\psi_n(t_n)\|_{\mathcal{F}L^p_{k}} \Tend n \infty \infty, \quad
  \forall  k\in \mathbb R,\ \forall
  p\in [1,\infty] .\]
\end{theorem}

\begin{remark}
   There is no general comparison between the assumptions on $s$ in
   Theorems~\ref{theo:inflationFLp} and \ref{theo:fourier-lebesgue}:
   for $p=1$, $\min(0,s_c(1))=-1/\si<-1/(2\si+1)$, while if $s_c(p)\ge
   0$, we obviously have $\min(0,s_c(p))=0>-1/(2\si+1)$.
 \end{remark}

\subsection{Modulation spaces}
\label{sec:modulation-spaces}

We now turn our attention to the theory of modulation spaces. 
The idea of modulation spaces is to  consider the decaying properties
of space variable and its Fourier transform
simultaneously. Specifically, we consider    the   short-time Fourier
transform (STFT)  (sliding-window transform/wave packet transform) of
$f$ with respect to   Schwartz class function $g\in
\mathcal{S}(\R^d)$: 
\begin{equation*}
V_{g}f(x,\xi)= \int_{\mathbb R^{d}} f(t) \overline{g(t-x)} e^{- i
  \xi\cdot t}dt, \quad (x, \xi) \in \mathbb R^{2d}, 
\end{equation*}
 whenever the integral exists.
Then the modulation spaces $M^{p,q}_s(\R^d)$  ($1 \leq p, q \leq
\infty$, $s\in \R)$ is defined as the collection of tempered
distributions $f\in \mathcal{S}'(\R^d)$  such
that
\[\|f\|_{M^{p,q}_s}= \left\|\|V_gf\|_{L^p_x}
    (1+|\xi|^2)^{s/2}\right\|_{L^q_{\xi}} < \infty,\]
with natural modification if a Lebesgue index is infinite.  
For $s=0,$ we write $M^{p,q}_0(\R^d)= M^{p,q}(\R^d).$ When $p=q=2,$
modulation spaces coincide   with usual Sobolev spaces $H^{s}(\R^d).$
For the last two decades, these spaces have made their own place in
PDEs and there is a  tremendous ongoing interest to use these
spaces as a   low regularity  Cauchy data class for nonlinear
dispersive  equations; see e.g. \cite{baoxiang2006isometric,
  benyi2009local, ruzhansky2012modulation,
  bhimani2019hartree,wang2007global, wang2011harmonic,
  OhWang2020}. Using  the algebra property and  boundedness of
Schr\"odinger propagator on $M^{p,q}_s(\R^d)$,  \eqref{nls}  is
proved to be locally well-posed in $M^{p,1}_s(\R^d)$ for $1\leq p  \leq \infty$,
$s\geq 0$,  and  in  $M^{p,q}_{s}(\R^d)$ for $1\leq p, q \leq \infty$
and  $s> d(1-1/q)$,
via fixed point argument; see  \cite{baoxiang2006isometric, benyi2009local,
  bhimani2016functions}.  Using uniform-decomposition techniques,
Wang and Hudzik  \cite{wang2007global}  established  global well-posedness
for \eqref{nls}  with small initial data in $M^{2,1}(\R^d).$ Guo
\cite{guo20171d} proved local well-posed  for the cubic NLS in  
$M^{2,q}(\R)$ $(2\leq q \leq \infty)$, and  later  Oh and Wang
\cite{OhWang2020},  established global existence  for this result. In
\cite{chaichenets2017existence},
Chaichenets et al.  established
global well-posedness for the cubic NLS in $M^{p, p'}(\R)$ for $p$
sufficiently close to $2$. The well-posedness problems for some other
PDEs in $M^{p,q}_s(\R^d)$ are widely studied by many authors, see for
instance the excellent survey  \cite{ruzhansky2012modulation}  and
references therein.  We complement the 
existing literature  on well-posedness theory for \eqref{nls} with
Cauchy data in modulation spaces. First, observe that, in view of
Proposition~\ref{ke} below,
\begin{equation*}
\left\| \psi \left (\lambda\, \cdot\right))\right\|_{M_s^{2,q}} \lesssim
\begin{cases}
\lambda^{- \frac{d}{2}} \max\(1,\lambda^s\)\|\psi\|_{M_s^{2,q}},  \quad \text{if} \quad  1\leq q \leq 2,\\
 \lambda^{-d \left( 1-\frac{1}{q} \right)} \max\(1,\lambda^s\)\|\psi\|_{M_s^{2,q}} ,\quad \text{if} \quad  2\leq q \leq \infty,
 \end{cases}
\end{equation*}
for all $  \lambda \leq 1$  and $s\in \R$. Invoking the general belief
that ill-posedness at positive regularity is due to the transfer from
low frequencies ($0<\lambda\ll 1$) to high frequencies, the scaling
\eqref{eq:scaling} suggests that ill-posedness occurs in $
M_s^{2,q}(\R^d)$ if
\begin{equation*}
  s<
  \begin{cases}
    s_c= \frac{d}{2}-\frac{1}{\si}\quad \text{if}\quad 1\le q\le 2,\\
    d\(1-\frac{1}{q}\)\quad \text{if}\quad 2\le q\le \infty.
  \end{cases}
\end{equation*}
The following analogue of Theorem~\ref{theo:inflationFLp} then appears
rather natural.
\begin{theorem}\label{theo:inflation-modulation}
  Let $  d, \sigma \in \mathbb N$ and  assume that 
\begin{itemize}
\item  $s< \min \(\frac{d}{2}-\frac{1}{\sigma}, 0 \)$ when $1\leq q \leq 2,$  and 
\item   $s< \min \(d\left(1- \frac{1}{q} \right)-\frac{1}{\sigma},
  0\)$  when $2\leq q \leq \infty.$ 
\end{itemize}
For any $\delta>0,$ there  exists  $\psi_0 \in M^{2,q}_s(\R^d)$  and $T>0$  satisfying
\[ \|\psi_{0} \|_{M^{2,q}_s} < \delta  \quad \text{and} \quad 0<T< \delta\] 
such  that  the corresponding solution $\psi$ to \eqref{nls} exists on $[0, T]$ and 
 \[\|\psi(T)\|_{M^{2,q}_s}>\delta^{-1}.\]
\end{theorem}

We also have some infinite loss of
regularity of the flow map \eqref{nls} at the level of modulation
spaces with negative regularity. We no longer assume $p=2$, and show a
stronger result, provided that the negative regularity $s$ is
sufficiently small, and (again) that we discard the one-dimensional cubic
case.
\begin{theorem}\label{theo:modulation}
  Let $\sigma\in \N$,  $s< -\frac{1}{2 \sigma
    +1}$ and assume $d\sigma \geq 2.$  There exists a  sequence of
  initial data  $\(\psi_n(0)\)_{n\geq 1}$ in  $\Sch(\R^d)$ such
  that  
\[ \|\psi_n(0)\|_{M_s^{p,q}} \Tend n \infty 0, \quad \forall p,q\in [1,\infty], \]
and a sequence of times  $t_n \to 0$ such that the corresponding solutions $\psi_n$ to \eqref{nls} satisfies
\[ \|\psi_n(t_n)\|_{M_{k }^{p,q}} \Tend n \infty \infty,
  \quad \forall k \in
  \mathbb R,\ \forall p,q\in [1,\infty].\] 
\end{theorem}

\begin{remark}
  Contrary to the Fourier-Lebesgue case, the assumption regarding $s$
  is always weaker in Theorem~\ref{theo:inflation-modulation} than in
  Theorem~\ref{theo:modulation} (recall that the cubic one-dimensional
  case is ruled out in  Theorem~\ref{theo:modulation}).
\end{remark}
\subsection{Comments and outline of the paper}

As pointed out before, the numerology regarding the norm inflation
phenomenon (Theorems~\ref{theo:inflationFLp} and
\ref{theo:inflation-modulation}) is probably sharp, up to the
fact that the minimum should be replaced by a maximum
in the assumption on $s$, and that  at positive regularity, different
arguments are required. On the other hand, we believe that the
restriction $s<-\frac{1}{2\sigma+1}$ in
Theorems~\ref{theo:fourier-lebesgue} and \ref{theo:modulation} is due
to  our approach, and  we expect  that the result is true under the
mere assumption  $s<0$ if $d\sigma\ge 2$, and for $s<-1/2$ if
$d=\sigma=1$. 
  
The analogue of  our results remains true  if we replace $\Delta$ by the
  generalized dispersion of the form 
$ \Delta_{\eta} = \sum_{j=1}^{d} \eta_{j}\partial^2_{x_j}, \eta_j=\pm1.$
The  \eqref{nls} associated $\Delta_{\eta}$ (with the non uniform
signs of  $\eta_j$)   arises in the description of surface gravity
waves on deep water, see e.g. \cite{SulemSulem}.

In  \cite{sugimoto2015local}, Sugimoto-Wang-Zhang established
  some local well-posedness results for  Davey-Stewartson equation in some
  weighted modulation spaces. We  note that  our method of proof can
  be applied to get  norm-inflation results for Davey-Stewartson
  equation,
  and infinite loss of regularity in
  the spirit of \cite{carles2012geometric},  in some negative
  modulation and Fourier-Lebesgue spaces.

Theorems~\ref{theo:fourier-lebesgue} and \ref{theo:modulation} cover
any smooth power nonlinearity in 
      multidimension, and power nonlinearities which are at least
      quintic in the one-dimensional case. Our method our proof seems
      too limited to prove loss of regularity in the case of
      the cubic nonlinearity on the line. It
      turns out that the method followed to treat the cubic
      nonlinearity on the circle in \cite{carles2017norm} seems
      helpless in the case of the line. On the other hand,
      Theorems~\ref{theo:inflationFLp} and
      \ref{theo:inflation-modulation} include the cubic
      one-dimensional Schr\"odinger equation.  

      \smallbreak

  The rest of this paper is organized as follows, In
  Section~\ref{sec:prelim}, we recall  various properties associated
  to modulation spaces. In Section~\ref{sec:inflationFLp}, we prove
  Theorem~\ref{theo:inflationFLp}, and we adapt the argument in
  Section~\ref{sec:inflation-modulation} to prove
  Theorem~\ref{theo:inflation-modulation}. In Section~\ref{sec:reduc},
  we show how the theory of weakly nonlinear geometric optics makes it
  possible to prove loss of regularity at negative regularity for
  \eqref{nls}. A general framework where multiphase weakly nonlinear
  geometric optics  is justified is presented in
  Section~\ref{sec:wnlgo}, and it is applied in
  Section~\ref{sec:conclusion} to prove
  Theorems~\ref{theo:fourier-lebesgue} and \ref{theo:modulation}. 

\subsubsection*{Notations}
The notation $A \lesssim B $ means $A \leq cB$ for a some constant $c > 0 $,
 Let $(\Lambda^{\eps})_{0<\eps \leq 1}$ and $(\Upsilon^{\eps})_{0<\eps \leq 1}$ be two families of  positive real numbers.
\begin{itemize}
\item  We write $\Lambda^{\eps}\ll \Upsilon^{\eps}$ if $\limsup_{\eps \to 0} \Lambda^{\eps}/\Upsilon^{\eps} =0.$
\item  We write $\Lambda^{\eps}\lesssim \Upsilon^{\eps}$ if $\limsup_{\eps \to 0} \Lambda^{\eps}/\Upsilon^{\eps} <\infty.$
\item  We write  $\Lambda^{\eps}\approx \Upsilon^{\eps}$  if  $\Lambda^{\eps}\lesssim \Upsilon^{\eps}$  and $\Upsilon^{\eps}\lesssim \Lambda^{\eps}$. 
\end{itemize}

\section{Preliminary: modulation spaces}\label{sec:prelim}  
 
 Feichtinger  \cite{feichtinger1983modulation} introduced a  class of
 Banach spaces,  the so-called modulation spaces,  which allow a
 measurement of space variable and Fourier transform variable of a
 function, or distribution, on $\mathbb R^d$ simultaneously, using the
 short-time Fourier transform (STFT). The  STFT  of a function $f$ with
 respect to a window function $g \in {\mathcal S}(\R^d)$ is defined by 
\begin{equation}\label{stft}
V_{g}f(x,y)= \int_{\mathbb R^{d}} f(t) \overline{g(t-x)} e^{- i y\cdot t}dt,  \  (x, y) \in \mathbb R^{2d},
\end{equation}
 whenever the integral exists.
For $x, y \in \R^d$, the translation operator $T_x$, and the modulation operator $M_y$, are
defined by $T_{x}f(t)= f(t-x)$ and $M_{y}f(t)= e^{ i y\cdot t} f(t).$ In terms of these
operators the STFT may be expressed as
\begin{equation}\label{fstft}
V_{g}f(x,y)=\langle f, M_{y}T_{x}g\rangle = e^{- i x \cdot w} \left(
  f\ast M_w g^*\right)(x)  ,
\end{equation}
 where $\langle f, g\rangle$ denotes the inner product for $L^2$ functions,
or the action of the tempered distribution $f$ on the Schwartz class function $g$, and $g^*(y) = \overline{g(-y)}.$  Thus $V: (f,g) \mapsto V_g(f)$ extends to a bilinear form on $\mathcal{S}'(\R^d) \times \mathcal{S}(\R^d)$, and $V_g(f)$ defines a uniformly continuous function on $\R^{d} \times \R^d$ whenever $f \in \mathcal{S}'(\R^d) $ and $g \in  \mathcal{S}(\R^d)$.
\begin{definition}[Modulation spaces]\label{ms} Let $1 \leq p,q \leq
  \infty$, $s \in \R$ and $0\neq g \in{\mathcal S}(\R^d)$. The  weighted  modulation space   $M_s^{p,q}(\R^d)$
is defined to be the space of all tempered distributions $f$ for which the following  norm is finite:
\[ \|f\|_{M_s^{p,q}}=  \left(\int_{\R^d}\left(\int_{\R^d}
      |V_{g}f(x,y)|^{p} dx\right)^{q/p} (1+|y|^2)^{sq/2} \,
    dy\right)^{1/q},\]
for $ 1 \leq p,q <\infty$. If $p$ or $q$ is infinite, $\|f\|_{M_s^{p,q}}$ is defined by replacing the corresponding integral by the essential supremum. 
\end{definition}
\begin{remark}
\label{equidm}
The definition of the modulation space given above, is independent of the choice of 
the particular window function.  See  \cite[Proposition 11.3.2(c)]{grochenig2013foundations}.
\end{remark}

We recall an alternative definition of modulation spaces via  the
frequency-uniform localization techniques, providing another
characterization which will be useful to prove
Theorem~\ref{theo:inflation-modulation}. Let  $Q_n$ be the unit cube with the center at  $n$, so $\( Q_{n}\)_{n \in \Z^d}$ constitutes a decomposition of  $\R^d,$ that is, $\R^d = \cup_{n\in \Z^{d}} Q_{n}.$
Let   $\rho \in \mathcal{S}(\R^d),$  $\rho: \R^d \to [0,1]$  be  a smooth function satisfying   $\rho(\xi)= 1 \  \text{if} \ \ |\xi|_{\infty}\leq \frac{1}{2} $ and $\rho(\xi)=
0 \  \text{if} \ \ |\xi|_{\infty}\geq  1.$ Let  $\rho_n$ be a translation of $\rho,$ that is,
\[ \rho_n(\xi)= \rho(\xi -n),\quad n \in \Z^d.\]
Denote 
\[\sigma_{n}(\xi)=
  \frac{\rho_{n}(\xi)}{\sum_{\ell\in\Z^{d}}\rho_{\ell}(\xi)}, \quad n
  \in \Z^d.\] 
Then   $\( \sigma_n(\xi)\)_{n\in \Z^d}$ satisfies the following properties:
\begin{equation}\label{fh}
  \left\{
    \begin{aligned}
      &|\sigma_{n}(\xi)|\geq c, \forall \xi \in Q_{n},\\
&\operatorname{supp} \sigma_{n} \subset \{\xi: |\xi-n|_{\infty}\leq 1 \},\\
&\sum_{n\in \Z^{d}} \sigma_{n}(\xi)\equiv 1, \forall \xi \in \R^d,\\
&|D^{\alpha}\sigma_{n}(\xi)|\leq C_{|\alpha|}, \forall \xi \in \R^d,
\alpha \in (\N \cup \{0\})^{d}. 
\end{aligned}
\right.
\end{equation}
The frequency-uniform decomposition operators can be exactly defined by 
\[\square_n = \mathcal{F}^{-1} \sigma_n \mathcal{F}. \]
For $1\leq p, q \leq \infty, s\in \R$,  it is known \cite{feichtinger1983modulation} that 
\begin{equation*}
\|f\|_{M^{p,q}_s}\asymp  \left( \sum_{n\in \mathbb Z^d} \left\lVert
  \square_n(f)\right\rVert^q_{L^p} (1+|n|)^{sq} \right)^{1/q}, 
\end{equation*}
with natural modifications for $p, q= \infty.$

\begin{lemma} [\cite{wang2011harmonic,
    grochenig2013foundations,ruzhansky2012modulation}]  \label{rl} Let
  $p,q, p_j, q_j\in [1, \infty]$, $s, s_j \in \R$ ($j=1,2$). Then
\begin{enumerate}
\item \label{ir} $M^{p_{1}, q_{1}}_{s_1}(\R^{d})
  \hookrightarrow M^{p_{2}, q_{2}}_{s_2}(\R^{d})$ whenever
  $p_{1}\leq p_{2}$ and $q_{1}\leq q_{2}$ and $s_2\leq s_1.$ In
  particular, $H^{s}(\R^d) \hookrightarrow M_s^{p,q}(\R^d)$ for $2\leq
  p,q \leq \infty$ and $s\in \R.$
  \item \label{uir}$M^{p_1, q_1}_{s_1} (\R^d) \hookrightarrow M^{p_2, q_2}_{s_2} (\R^d)$ for $q_1> q_2, s_1>s_2$ and $s_1-s_2> d/q_2-d/q_1.$ 
%\item \label{dh} $\mathcal{F}L^q(\R^d) \subset M^{2, q}(\R^d) $ for $2\leq q \leq \infty.$
\item \label{el} $M^{p,q_{1}}(\R^{d}) \hookrightarrow L^{p}(\R^{d}) \hookrightarrow M^{p,q_{2}}(\R^{d})$ holds for $q_{1}\leq \text{min} \{p, p'\}$ and $q_{2}\geq \text{max} \{p, p'\}$ with $\frac{1}{p}+\frac{1}{p'}=1.$
\item \label{rcs} $M^{\min\{p', 2\}, p}(\R^d) \hookrightarrow \mathcal{F} L^{p}(\R^d)\hookrightarrow M^{\max \{p',2\},p}(\R^d),  \frac{1}{p}+\frac{1}{p'}=1.$
\item $\mathcal{S}(\R^{d})$ is dense in  $M^{p,q}(\R^{d})$ if $p$ and $q$ are finite. 
\item  $M^{p,p}(\R^d) \hookrightarrow L^p(\R^d) \hookrightarrow M^{p,p'}(\R^d)$ for $1\leq p \leq 2$ and  $M^{p,p'}(\R^d)  \hookrightarrow L^p(\R^d) \hookrightarrow M^{p,p}(\R^d)$ for $2 \leq p \leq \infty.$
\item \label{fi} The Fourier transform $\mathcal{F}:M_s^{p,p}(\R^{d})\to M_s^{p,p}(\R^{d})$ is an isomorphism.
\item The space  $M_s^{p,q}(\R^{d})$ is a  Banach space.
\item \label{ic}The space $M_s^{p,q}(\R^{d})$ is invariant under
  complex conjugation. 
\end{enumerate}
\end{lemma}

\begin{theorem}[Algebra property] \label{theo:algebra} Let $p,q, p_{i}, q_{i}\in [1, \infty]$  $(i=0,1,2).$
If   $\frac{1}{p_1}+ \frac{1}{p_2}= \frac{1}{p_0}$ and $\frac{1}{q_1}+\frac{1}{q_2}=1+\frac{1}{q_0}, $ then
\begin{equation}\label{prm}
M^{p_1, q_1}(\mathbb R^{d}) \cdot M^{p_{2}, q_{2}}(\mathbb R^{d}) \hookrightarrow M^{p_0, q_0}(\mathbb R^{d});
\end{equation}
with norm inequality $\|f g\|_{M^{p_0, q_0}}\lesssim \|f\|_{M^{p_1, q_1}} \|g\|_{M^{p_2,q_2}}.$
In particular, the  space $M^{p,q}(\mathbb R^{d})$ is a pointwise
$\mathcal{F}L^{1}(\mathbb R^{d})$-module, that is, we have
\begin{equation*}\label{mp}
\|fg\|_{M^{p,q}} \lesssim \|f\|_{\mathcal{F}L^{1}} \|g\|_{M^{p,q}}.
\end{equation*} 
\end{theorem}
\begin{proof}
The product relation \eqref{prm} between modulation spaces is well known and we refer the interested reader to  \cite{benyi2009local}  and  since $\mathcal{F}L^{1}(\mathbb R^{d}) \hookrightarrow M^{\infty, 1}(\mathbb R^{d})$, the desired inequality \eqref{mp} follows.
\end{proof}

For $f\in \mathcal{S}(\mathbb R^{d}),$ the Schr\"odinger propagator
$e^{i\frac{t}{2}\Delta}$ is given by
\begin{equation*}
\label{sg}
e^{i\frac{t}{2}\Delta}f(x)= \frac{1}{(2\pi)^d}\int_{\mathbb R^d}
e^{ix\cdot \xi}e^{-i\frac{t}{2}|\xi|^2}\hat f(\xi) d\xi.
\end{equation*}
The first point in the following statement was established in
\cite{benyi2007unimodular}, and the second, in \cite[Proposition
4.1]{wang2007global}. 
\begin{proposition}[\cite{benyi2007unimodular, wang2007global}]\label{uf} \
\begin{enumerate}
\item \label{e1} Let $t\in \mathbb R, \  p, q\in [1, \infty].$ Then
\begin{equation*}
\|e^{i\frac{t}{2} \Delta} f\|_{M^{p,q}} \leq C(t^{2}+1)^{d/4}\|f\|_{M^{p,q}}
\end{equation*}
where $C$ is some constant depending on $d.$
\item \label{e2} Let  $ 2 \leq p \leq \infty, 1\leq q \leq \infty.$
  Then
  \[\|e^{i\frac{t}{2}\Delta} f \|_{M^{p,q}}\leq  (1+|t|)^{- d\left( \frac{1}{p}-\frac{1}{2} \right)} \|f\|_{M^{p',q}}.\]
\end{enumerate}
\end{proposition}

For $(1/p, 1/q)\in [0,1] \times [0,1],$ we define the subsets 
\begin{align*}
&I_1=\{(p,q);\ \max (1/p,1/p')\leq 1/q\},\quad I_1^{*}=\{(p,q);\ \min (1/p,1/p')\geq 1/q\},\\
&I_2=\{(p,q);\ \max (1/q,1/2)\leq 1/p'\},\quad I_2^{*}=\{(p,q);\ \min (1/q,1/2)\geq 1/p'\},\\
&I_3=\{(p,q);\ \max (1/q,1/2)\leq 1/p\}, \quad I_3^{*}=\{(p,q);\ \min (1/q,1/2)\geq 1/p\}.
\end{align*}	
We now define the indices:
\begin{equation*}
\mu_1(p,q)=\begin{cases}
-1/p & \text{ if }   (1/p,1/q)\in I_1^{*},\\
1/q-1 & \text{ if }  (1/p,1/q)\in I_2^{*},\\
-2/p+1/q & \text{ if }  (1/p,1/q)\in I_3^{*},
\end{cases}
\end{equation*}
and 
\begin{equation*}
\mu_2(p,q)=\begin{cases}
-1/p  & \text{ if }   (1/p,1/q)\in I_1,\\
1/q-1 & \text{ if }  (1/p,1/q)\in I_2,\\
-2/p+1/q & \text{ if }  (1/p,1/q)\in I_3.
\end{cases}
\end{equation*}
The dilation operator $f_{\lambda}$ is given by
\[ f_{\lambda}(x)= f(\lambda x), \quad \lambda >0.\]
\begin{proposition}[See Theorem 3.2 in \cite{cordero2012dilation}]\label{ke}
Let $1\leq p, q \leq \infty$, $s\in \R.$ There exists  a constant
$C>0$ such that for all $f\in M^{p,q}_{s}(\R^d), 0< \lambda \leq 1,$
we have  
\[C^{-1} \lambda^{d \mu_1(p,q)} \min \{1, \lambda^s \}
  \|f\|_{M^{p,q}_s} \leq  \|f_{\lambda} \|_{M^{p,q}_s} \leq C
  \lambda^{d \mu_2(p,q)} \max\{1, \lambda^s\} \|f\|_{M^{p,q}_s}.\] 
\end{proposition}

\section{Norm inflation in Fourier-Lebesgue spaces}
\label{sec:inflationFLp}
Define
\[ \mu_{\sigma} (z_1,\dots,z_{2\sigma +1})= \prod_{\ell=1}^{\sigma +1} z_{\ell} \prod_{m=\sigma+2}^{2\sigma +1} \bar{z}_m.\]
\begin{definition}\label{imd} For $\psi_0\in L^2(\R^d),$  define
$U_1[\psi_0](t)= e^{i \frac{t}{2} \Delta} \psi_0,$
\[ U_k[\psi_0](t) = -i \sum_{{k_1, \dots, k_{2\sigma+1} \geq 1}\atop{k_1+\dots+k_{2\sigma +1} =k}}\int_0^t e^{i \frac{(t-\tau)}{2} \Delta}\mu_{\sigma}\left(  U_{k_1}[\psi_0],...,U_{k_{2\sigma +1}}[\psi_0] \right)(\tau) d\tau,  \quad k\geq 2.\]
\end{definition}
It is known  that  the solution $\psi$ of \eqref{nls} can be written
as a  power series  expansion $\psi= \sum_{k=1}^{\infty}
U_k[\psi_0]$, see \cite{bejenaru2006sharp}, and \cite{IwabuchiOgawa,kishimoto2019}
for later refinements of the method. 

\begin{definition}  Let $A>0$ be a dyadic number.  Define the  space $M_{A}$ as the completion
of $C_0^{\infty}(\R^d)$ with  respect to the norm
\[\|f\|_{M_{A}}= \sum_{\xi \in A\Z^d}  \|\hat{f}\|_{L^2 (\xi+
    Q_A)},\quad Q_A= [-A/2, A/2)^d.\]
\end{definition}
\begin{lemma}[\cite{bejenaru2006sharp, kishimoto2019}]\label{nkm}
  Let $A>0$ be a dyadic number.
  \begin{enumerate}
\item $M_A \sim_A M_1$, and for all $\epsilon>0$, $H^{\frac{d}{2}+ \epsilon} \hookrightarrow  M_1 \hookrightarrow  L^2 $.
\item  $M_A$ is a Banach algebra under pointwise multiplication, and
\[ \|fg\|_{M_A} \leq C(d) A^{d/2} \|f\|_{M_A} \|g\|_{M_A} \quad
  \forall f, g \in M_A.\]
\item\label{3}  Let $A \geq 1$ be a dyadic number and $\phi \in M_A$ with $\|\psi_0\|_{M_A} \leq M.$ Then, there exists $C>0$ independent of $A$ and $M$ such that 
\[ \|U_k[\psi_0] (t) \|_{M_A} \leq t^{\frac{k-1}{2\sigma}} (C A^{d/2} M)^{k-1} M ,\]
for any $t\geq 0$ and $k \geq 1.$
\item  Let $\(b_k\)_{k=1}^{\infty}$ be a sequence of nonnegative  real numbers such that 
\[ b_k \leq C  \sum_{{k_1,\dots, k_{2\sigma +1} \geq
      1}\atop{k_1+\dots+ k_{2\sigma +1} =k}}  b_{k_1}\cdots
  b_{k_{2\sigma +1}} \quad \forall k \geq 2.\]
Then  we have 
\[ b_k \leq b_1 C_0^{k-1}, \quad\forall k \geq 1,\text{ where } C_0=
  \frac{\pi^2}{6} \(C (2\sigma +1)^2\)^{1/(2\sigma)} b_1.\] 
\end{enumerate}
\end{lemma}
\begin{corollary}[See Corollary 1 in \cite{kishimoto2019}] \label{uc}Let $A\geq 1$ be dyadic and $M>0$. If $0<T\ll (A^{d/2} M)^{-2\sigma},$ then  for any  $\psi_0 \in M_A$ with  $\|\psi_0\|_{M_A} \leq M$:\\
$(i)$ A unique solution $\psi$ to the integral equation associated with \eqref{nls},
\[\psi(t) = e^{i \frac{t}{2} \Delta} \psi_0 - i \int_0^t e^{i\frac{(t-\tau)}{2} \Delta} \mu_{\sigma} (\psi(\tau)) d\tau\]
exists in $C([0, T], M_A).$ \\
$(ii)$ The solution $\psi$ given in $(i)$ has the expression
\begin{equation}
  \label{eq:Picard-sum}
  \psi= \sum_{k=1}^{\infty} U_k[\psi_0]= \sum_{\ell =0}^{\infty} U_{2\sigma \ell +1}[\psi_0] 
\end{equation}
which converges absolutely in $C([0, T], M_A).$
\end{corollary}
\begin{remark}
 By  Definition \ref{imd},  we obtain  $U_k[\psi_0](t)=0$ unless
 $k\equiv 1$ mod $2\sigma$.  For instance, $U_k[\psi_0](t) \equiv 0$ for all $k \in 2\sigma  \N.$ To see this,  fix $\sigma \in \N.$ Then clearly $U_{2\sigma}[\psi_0]\equiv 0$ because  there does not exist $k_j\geq 1$ such that $k_1+\cdots + k_{2\sigma +1}=2\sigma.$ Now since $ U_{2\sigma }[\psi_0] \equiv 0,$ it follows that $U_{4\sigma}[\psi_0] \equiv 0$ and so on. Thus,  $U_k[\psi_0](t) \equiv 0$ for all $k \in 2\sigma  \N.$
 \end{remark}
The general idea from \cite{bejenaru2006sharp} to prove instability is
to show that one term in the sum \eqref{eq:Picard-sum} dominates the
sum of the other terms, and rules out the continuity of the flow
map. Usually, the first Picard iterate accounting for nonlinear
effects, that is, $U_{2\si+1}[\psi_0]$ in our case, does the job. The
proof of Theorems~\ref{theo:inflationFLp} and
\ref{theo:inflation-modulation} indeed relies on this idea, for a
suitable $\psi_0$ as in \cite{kishimoto2019}.
\smallbreak

Let $N, A$  be dyadic numbers to be specified so that   $N\gg 1$ and $0<A\ll N$. We choose initial data  of the following form
\begin{equation}\label{kid}
\widehat{\psi_0} = R A^{-d/p} N^{-s} \chi_{\Omega},
\end{equation}
for a positive constant $R$ and a set $\Omega$ satisfying 
\[ \Omega = \bigcup_{\eta \in \sum} (\eta + Q_A) ,\]
for some $\sum \subset \{ \xi \in \R^d:  |\xi| \sim N\}$ such that $\#
\sum \leq 3.$ Then we have  
\[ \| \psi_0\|_{\mathcal{F}L^p_s} \sim R,  \quad \|\psi_0\|_{M_A} \sim  R A^{d (\frac{1}{2}-\frac{1}{p})} N^{-s}.\]
In fact, we have 
\begin{align*}
\|\psi_0\|_{\mathcal{F}L^p_s}^p &= R^p A^{-d} N^{-sp} \int_{\Omega}
                                  (1+|\xi|^2)^{ps/2} d\xi\\
  &= R^p A^{-d} N^{-sp} \sum_{\eta} \int_{\eta + Q_A} (1+|\xi|^2)^{ps/2} d\xi.
\end{align*}
Since  $A<N$ and $|\eta| \sim  N,$ we have $N^{2} \lesssim  (1+|\xi|^2) \lesssim N^2$ for $\xi \in \eta +Q_A$ and so $N^{ps} \lesssim  (1+|\xi|^2)^{ps/2} \lesssim N^{ps}$.    As $\# \sum \leq 3$ and $|\eta + Q_A| \sim A^{d},$ we infer that $ \|\psi_0\|_{\mathcal{F}L^p_s}^p \sim  R^p.$

\begin{lemma}[See Lemma 3.6 in \cite{kishimoto2019}]\label{kms}There exists $C>0$ such that for any $\psi_0$ satisfying \eqref{kid} and $k\geq 1$, we have 
\[\left|  \operatorname{supp}\widehat{ U_k[\psi_0]} (t)\right| \leq C^k
  A^d ,\quad \forall t\geq 0.\]
\end{lemma}
The next result is the analogue of \cite[Lemma~3.7]{kishimoto2019}.
\begin{lemma} \label{d1} Let $\psi_0$ given by \eqref{kid}, $s<0$ and
  $1\leq p \leq  \infty.$ Then there exists $C>0$ depending  only on
  $d, \sigma$ and $s$ such that following holds.
  \begin{align}
  &  \|U_1[\psi_0](T)\|_{\mathcal{F}L^p_s} \leq C R ,\quad \forall
    T\ge 0,\label{de1}   \\
    &\| U_k[\psi_0](T)\|_{\mathcal{F}L^p_s} \lesssim \rho_1^{k-1}
      C^{k} A^{-d/p} RN^{-s} \|\langle \cdot \rangle^s\|_{L^p
      (Q_A)},\label{de2} 
  \end{align}
  where $\rho_1=  R N^{-s} A^{d \left( 1- \frac{1}{p} \right)}
  T^{\frac{1}{2\sigma}}$. 
\end{lemma}
\begin{proof} The Schr\"odinger group is a Fourier multiplier,
\begin{equation*}
\| U_1[\psi_0](T)\|_{\mathcal{F}L^p_s}  =  \left\| (e^{i\frac{t}{2}|\cdot|^2} \widehat{\psi_0} )\langle \cdot \rangle^s \right\|_{L^p} = \|\psi_0\|_{\mathcal{F}L^p_s} \leq CR, 
\end{equation*}
hence \eqref{de1}. We note that 
\begin{align*}
I& :=  \|U_k[\psi_0](T)\|_{\mathcal{F}L^p_s} \\
& \leq  \| \langle \cdot \rangle^s \|_{L^p( \operatorname{supp}\widehat{U_k}[\psi_0](t))} \sup_{\xi \in \R^d} \left|  \widehat{U_k[\psi_0]} (t, \xi)\right|\\
& \leq   \| \langle \cdot \rangle^s \|_{L^p( \operatorname{supp}
 \widehat{U_k}[\psi_0](t))} \sum_{k_1+\dots + k_{2\sigma +1} =k} \int_0^t \left\| |v_{k_1} (\tau)|\ast \dots \ast |v_{k_{2\sigma +1}}(\tau)| \right\|_{L^{\infty}} d\tau ,
\end{align*}
where $v_{k_{\ell}}$ is either $\widehat{U_{k_{\ell}}[\psi_0]}$
or $\widehat{\overline{ U_{k_{\ell}}[\psi_0]}}$.
By Young  and Cauchy-Schwarz inequalities,  
\begin{align*}
\|v_{k_1}\ast \cdots \ast v_{k_{2\sigma +1}}\|_{L^{\infty}} & \leq
   \|v_{k_1}\ast v_{k_2}\|_{L^{\infty}}\|v_{k_3} \ast\cdots
  \ast v_{k_{2\sigma +1}}\|_{L^1}\\
   &\leq  \|v_{k_1}^{\vee} v_{k_2}^{\vee}\|_{L^1} \prod_{\ell=3}^{2\sigma +1} \|v_{k_{\ell}}\|_{L^1}\\
  & \leq  \|v_{k_1} \|_{L^2} \| v_{k_2}\|_{L^2}\prod_{\ell=3}^{2\sigma+1} \|v_{k_{\ell}}\|_{L^1}\\
  &\leq \prod_{\ell=3}^{2\sigma +1} \left| \operatorname{supp}  \widehat{ U_{k_{\ell}}[\psi_0]} \right|^{1/2} \prod_{\ell=1}^{2\sigma +1} \| \widehat{ U_{k_{\ell}}[\psi_0]} \|_{L^2}.
\end{align*}
Thus, we  have 
\begin{equation*}
I  \leq   \| \langle \cdot \rangle^s \|_{L^p(
  \operatorname{supp}\widehat{U_k}[\psi_0](t))} I_1,
\end{equation*}
where
\begin{equation*}
  I_1:=\sum_{k_1+\cdots + k_{2\sigma +1} =k} \int_0^t \prod_{\ell=3}^{2\sigma +1} \left| \operatorname{supp} \widehat{ U_{k_{\ell}}[\psi_0]} (\tau) \right|^{1/2} \prod_{\ell=1}^{2\sigma +1} \| \widehat{ U_{k_{\ell}}[\psi_0]}(\tau) \|_{L^2}  d\tau.
\end{equation*}
 By Lemma \ref{nkm} \eqref{3}  (with $M= CR N^{-s} A^{\frac{d}{2}-\frac{d}{p}}$), we  have, for all $k\ge 1$,   
\[ \| U_{k}[\psi_0](t)\|_{L^2} \leq  \| U_{k}[\psi_0](t)\|_{M_A}  \leq  Ct^{\frac{k-1}{2\sigma}} \left( C^2R A^{d/2} N^{-s}  A^{\frac{d}{2}-\frac{d}{p}}\right)^{k-1} R N^{-s} A^{\frac{d}{2}- \frac{d}{p}}. \]
Note that, by Lemma \ref{kms}, 
\begin{align*}
I_1& \lesssim   \sum_{k_1+\dots + k_{2\sigma +1} =k} \int_0^t \prod_{\ell=3}^{2\sigma +1}  A^{d/2} \prod_{\ell=1}^{2\sigma +1}\left[ \tau^{\frac{k_{\ell}-1}{2\sigma}} \left( R A^{d \left(1-\frac{1}{p} \right)} N^{-s} \right)^{k_{\ell}-1} R N^{-s} A^{\frac{d}{2}-\frac{d}{p}}\right]  d\tau\\
& \lesssim    (RN^{-s})^{k}  A^{ \frac{d(2\sigma-1)}{2}} A^{d \left( 1- \frac{1}{p} \right) (k-2\sigma -1)} A^{\left( \frac{d}{2}- \frac{d}{p}\right) (2\sigma +1)} \int_0^t \tau^{\frac{k-2\sigma -1}{2\sigma}} d\tau \\
& \lesssim    A^{d \left( 1-\frac{1}{p} \right) (k-1)}  A^{-d/p}(RN^{-s})^{k} t^{\frac{k-1}{2\sigma}}.
\end{align*}
Since $s<0$, for any bounded set $D\subset \R^d$, we have 
\[ \left|  \{  \langle \xi \rangle^{s} > \lambda \} \cup D \right| \leq \left|  \{  \langle \xi \rangle^{s} > \lambda \} \cup B_D \right|,  \quad \forall\lambda>0,\]
where $B_D\subset \R^d$ is the ball centered at origin with $|D|=|B_D|.$  This implies that  $\|\langle \xi \rangle^s\|_{L^p(D)} \leq \|\langle \xi \rangle^s\|_{L^p(B_D)}.$ In view of this and  performing simple change of variables ($\xi= C^{k/d} \xi'$), we obtain 
\begin{equation*}
\| \langle \cdot \rangle^s \|_{L^p( \operatorname{supp}
  \widehat{U_k}[\psi_0](t))}\leq
\| \langle \cdot \rangle^s \|_{L^p ( \{ |\xi| \leq C^{k/d} A \})}
\lesssim   C^{k} \|\langle \cdot \rangle^s \|_{L^p(\{ |\xi| \leq A\})},
\end{equation*}
and the lemma follows.
\end{proof}
In the next lemma we establish a crucial lower bound on $U_{2\sigma+1}[\psi_0].$

\begin{lemma}\label{d2}  Let $1\leq p \leq  \infty,$  $1\leq A \ll N$ and $\sum=\{ Ne_d, - Ne_d, 2Ne_d\}$ where $e_d= (0,\dots,0, 1) \in \R^d$. If $0<T\ll N^{-2},$ then we   have 
\[ \|U_{2\sigma +1} [\psi_0] (T)\|_{\mathcal{F}L_s^p}  \gtrsim  R A^{-\frac{d}{p}} N^{-s} \rho_1^{2\sigma}\|\langle \cdot \rangle^s\|_{L^p (Q_A)}, \]
where  $\rho_1=  R N^{-s} A^{d- \frac{d}{p}} T^{\frac{1}{2\sigma}}$.
\end{lemma}
\begin{proof}
Note that  
\begin{equation*}
\widehat{U_{2\sigma +1}[\psi_0]}(T, \xi)= c e^{-i \frac{T}{2} |\xi|^2} \int_{\Gamma} \prod_{\ell=1}^{\sigma +1} \widehat{\psi_0} (\xi_{\ell}) \prod_{m=\sigma +2}^{ 2\sigma +1} \overline{ \widehat{\psi_0} (\xi_m)} \int_0^{T} e^{i\frac{t}{2} \Phi} dt d\xi_1...d\xi_{2\sigma +1},
\end{equation*}
where
\begin{align*}
  &\Gamma= \left \{  (\xi_1,\dots, \xi_{2\sigma +1})\in
    \R^{(2\sigma+1)d}:  \sum_{\ell=1}^{\sigma +1} \xi_{\ell} -
    \sum_{m= \sigma +2}^{2\sigma +1} \xi_m = \xi \right \}, \\
  &\Phi = |\xi|^2- \sum_{\ell=1}^{\sigma +1}|\xi_{\ell}|^2 +  \sum_{m=
    \sigma +2}^{2\sigma +1} |\xi_m|^2.  
\end{align*}
By  the choice of initial data \eqref{kid}, we  have 
\begin{align*}
 \int_{\Gamma} \prod_{\ell=1}^{\sigma +1} \widehat{\psi_0} (\xi_{\ell}) \prod_{m=\sigma +2}^{ 2\sigma +1} \overline{ \widehat{\psi_0} (\xi_m)}  & =   \int_{\Gamma} \prod_{\ell=1}^{\sigma +1}  R A^{-d/p} N^{-s} \chi_{\Omega} (\xi_{\ell}) \prod_{m=\sigma +2}^{ 2\sigma +1}  R A^{-d/p} N^{-s} \chi_{\Omega} (\xi_{m})  \\
 & =  \left( R A^{-d/p} N^{-s} \right)^{2\sigma +1} \int_{\Gamma}
   \prod_{\ell=1}^{2\sigma +1} \chi_{\Omega} (\xi_{\ell}) d\xi_1 \dots
   d\xi_{2\sigma +1}\\
 & =    \left( R A^{-d/p} N^{-s} \right)^{2\sigma +1}
   \sum_{\mathcal C} \int_{\Gamma}
   \prod_{\ell=1}^{2\sigma +1} \chi_{\eta_{\ell} + Q_{A}} (\xi_{\ell})
   d\xi_1 \dots d\xi_{2\sigma +1},
\end{align*}
where the sum is taken over the non-empty set
\[ \mathcal C =\left\{  (\eta_1, \dots, \eta_{2\sigma +1}) \in \{ \pm Ne_d, 2 Ne_d\}^{2\sigma +1}:  \sum_{\ell=1}^{\sigma+1} \eta_{\ell} - \sum_{m= \sigma +2}^{2\sigma +1} \eta_{m} =0 \right\} .\]
For $\xi \in Q_{A},$ we  have $|\xi_{i}|^2\leq |\xi|^2 \leq A^2 \ll
N^2$ and so  $\left| \Phi \right| \lesssim N^2.$ Then $|
\frac{t}{2}\Phi (\xi)|\ll 1$ for $0<T \ll N^{-2}.$  In view of this,
together with the  fact that  the  cosine function
decreasing on $[0, \pi/4],$ we obtain 
\begin{equation*}
\left| \int_0^{T} e^{i\frac{t}{2} \Phi(\xi)} dt\right| \geq  \RE\int_0^{T} e^{i\frac{t}{2} \Phi(\xi)} dt \geq  \frac{1}{2} T.
\end{equation*}
Taking the above inequalities   into account, we infer
\begin{equation}\label{lu}
\left| \widehat{U_{2\sigma +1}[\psi_0]}(T, \xi)\right| \gtrsim  \left( R A^{-d/p} N^{-s} \right)^{2\sigma +1}  (A^d)^{2\sigma} T \chi_{(2\sigma +1)^{-1} Q_{A}} (\xi).
\end{equation}
Hence, we  have 
\begin{align*}
\| U_{2\sigma +1}[\psi_0] (T) \|_{\mathcal{F}L^p_s} & \gtrsim   \left( R A^{-d/p} N^{-s} \right)^{2\sigma +1}  (A^d)^{2\sigma} T \|\langle \cdot \rangle^{s}\|_{L^p ( (2\sigma +1)^{-1} Q_A)}\\
& \gtrsim  R A^{-\frac{d}{p}} N^{-s} \rho_1^{2\sigma} \|\langle \cdot \rangle^s\|_{L^p (Q_A)},
\end{align*}
where $\rho_1=  R N^{-s} A^{d- \frac{d}{p}} T^{\frac{1}{2\sigma}}$.
\end{proof}
For the convenience of  reader, we compute  the  $L^p$-norm of  weight
$ \langle \cdot \rangle^s $ on the cube $Q_A.$ 
\begin{lemma}\label{dnl} Let $A\gg 1$, $d\ge 1$, $s<0$ and $1\leq p <
  \infty$. We define 
\begin{equation*} f^p_s(A)= 
\begin{cases} 1 \quad \text{if} \ s< -\frac{d}{p},\\
\left( \log A \right)^{1/p} \quad \text{if} \ s= -\frac{d}{p},\\
A^{d/p+s} \quad \text{if} \  s> -\frac{d}{p}.
\end{cases}
\end{equation*}
Then we have 
$f^p_s(A) \lesssim \left\|\langle \cdot \rangle^s \right\|_{L^p(Q_A)} \lesssim f^p_s(A)$ and  $f_s^{\infty}(A)= \left\|\langle \cdot \rangle^s \right\|_{L^{\infty}(Q_A)} \sim 1. $ 
In particular,  $f^{p}_s(A) \gtrsim  A^{\frac{d}{p}+s}$ for any $s<0.$
\end{lemma}
\begin{proof}
We first compute the  $\|\cdot\|_{L^{p}}$-norm on  ball of  radius $R_1$ in $\R^d$, say $B_{R_1}(0).$   Since $\langle \cdot \rangle^s$ is radial, we  have 
\[I(R_1):= \int_{B_{R_1}(0)} \frac{1}{(1+|\xi|^2)^{-sp/2}}  d\xi =   \frac{2 \pi^{d/2}}{\Gamma(d/2)} \int_0^{R_1} \frac{r^{d-1}}{(1+r^2)^{-sp/2}} dr.\]
Notice that $(1+r^2)^{-sp/2} \geq \max\{1,r^{-sp}\}$, and assuming that $R_1 \gg 1$, we obtain:
\begin{align*}
I(R_1)  & \lesssim   \int_0^1 \frac{r^{d-1}}{\max\{1,r^{-sp}\}} \ \text{d} t + \int_1^{R_1} \frac{r^{d-1}}{\max\{1,r^{-sp}\}}  dr \\
& =    \int_0^1 r^{d-1} dr + \int_1^{R_1} \frac{1}{r^{-sp-d+1}}  dr.
\end{align*}
Using conditions on $s,$ we have $I(R_1) \lesssim \left(
  f_s^p(R_1)\right)^{p}.$ Notice that $Q_{A} \subset
B_{\sqrt{d}A/2}(0),$ we have  $\| \langle \cdot \rangle^s \|_{L^p
  (Q_A)} \leq  \left( I (\sqrt{d} A/2) \right)^{1/p} \lesssim f_s^p(A).$
On the other hand, we notice that $1+r^2 \leq 2$ if $0<r<1$ and $1+r^2
\leq 2r^2$ if $1<r< R_2$ for  some appropriate $R_2$. Using this
together with the above ideas, we obtain  $ f_s^p(A) \lesssim \| \langle
\cdot \rangle^s \|_{L^p (Q_A)}.$ 
This completes the proof.
\end{proof}
\begin{proof}[Proof of Theorem \ref{theo:inflationFLp}]   By
  Corollary~\ref{uc}, we have  the existence of a unique solution  to
  \eqref{nls} in $M_{A}$  up to time  $T$ whenever  $\rho_1=RN^{-s}
  A^{d \left( 1- \frac{1}{p}\right)} T^{1/(2\sigma)} \ll 1$.  In view
  of  Lemma~\ref{d1} and since $\rho_1<1$,  $\sum_{\ell =2}^{\infty}
  \| U_{2\sigma \ell +1}[\psi_0] (T)\|_{\mathcal{F}L^{p}_s} $ can be
  dominated by the  sum of the geometric series. Specifically, we have  
\begin{align}\label{mv1}
\left\| \sum_{\ell =2}^{\infty} U_{2\sigma \ell +1}[\psi_0] (T) \right\|_{\mathcal{F}L^p_s} & \lesssim   A^{-d/2} R N^{-s} f_s^p(A) \sum_{\ell=2}^{\infty} \rho_1^{2\sigma \ell} \nonumber \\
& \lesssim  A^{-d/2} R N^{-s} f_s^p(A)\rho_1^{4\sigma}.
\end{align}
By Corollary \ref{uc} and the triangle  inequality,  we  obtain 
\begin{align*}
\|\psi(T)\|_{\mathcal{F}L^p_s} & = \left\| \sum_{\ell =0}^{\infty} U_{2\sigma \ell +1}[\psi_0] \right\|_{\mathcal{F}L^p_s}  \\
& \geq  \| U_{2\sigma +1}[\psi_0](T)\|_{\mathcal{F}L^p_s} - \|U_1[\psi_0](T)\|_{\mathcal{F}L^p_s}- \left\| \sum_{\ell =2}^{\infty} U_{2\sigma \ell +1}[\psi_0](T) \right\|_{\mathcal{F}L^p_s}.
\end{align*}
 In order to ensure 
\[ \|\psi(T)\|_{\mathcal{F}L^p_s} \gtrsim  \| U_{2\sigma +1}[\psi_0](T)\|_{\mathcal{F}L^p_s},   \]
we rely on the conditions 
\begin{align}\label{rc1}
 &\| U_{2\sigma +1}[\psi_0](T)\|_{\mathcal{F}L^p_s}  \gg
  \|U_1[\psi_0](T)\|_{\mathcal{F}L^p_s},\\
\label{rc2}
&  \| U_{2\sigma +1}[\psi_0](T)\|_{\mathcal{F}L^p_s}  \gg  \left\| \sum_{\ell =2}^{\infty} U_{2\sigma \ell +1}[\psi_0](T) \right\|_{\mathcal{F}L^p_s}.
\end{align}
To use Lemma \ref{d2}, we  require
\begin{itemize}
\item[(i)]  $T\ll N^{-2}.$
\end{itemize}
In view of Lemma \ref{d1},  to prove \eqref{rc1} it is sufficient to prove 
\begin{itemize}
\item[(ii)]  $R \rho_1^{2\sigma} A^{-\frac{d}{p}} N^{-s} f^p_s(A) \gg
  R$, with $\rho_1=  R N^{-s} A^{d- \frac{d}{p}} T^{\frac{1}{2\sigma}}$.
\end{itemize}
Finally, in view  of Lemmas \ref{d1}, \ref{d2} and \ref{dnl}, and \eqref{mv1},   to prove \eqref{rc2}  it is sufficient to prove:\begin{itemize}
\item[(iii)] $\rho_1\ll 1$,
\item[(iv)] $R \rho_1^{2\sigma} A^{-\frac{d}{p}} N^{-s} f^p_s(A) \gg R
 \rho_1^{4\sigma} A^{-\frac{d}{p}}  N^{-s} f_s^p(A).$
\end{itemize}
We now choose $A, R$ and $T$ so that conditions  (i)- (iv) are
satisfied. To this end,  we set  
\[ R= (\log N)^{-1}, \quad A \sim (\log N)^{-\frac{2\sigma +2}{|s|}} N, \quad  T= (A^{d(\frac{1}{p}-1)} N^{s})^{2\sigma}. \]
Then we have
\[\rho_1= R N^{-s} A^{d- \frac{d}{p}} T^{\frac{1}{2\sigma}}=(\log N)^{-1}\ll 1.\]
  Hence, condition  (iii) is satisfied and so condition (iv). Note that
\[T=  (\log N)^{-\frac{2\sigma +2}{|s|} d (\frac{1}{p} -1) 2\sigma} N^{d (\frac{1}{p} -1) 2\sigma + 2\sigma s}.\]
Since $s< d \left( 1-\frac{1}{p} \right)- \frac{1}{\sigma}$ and $\log N = \mathcal{O}(N^{\epsilon})$ for any $\epsilon>0,$ we   have 
\[ T \ll N^{-2}, \]
and hence  (i) is satisfied.
By Lemma \ref{dnl}, we have  $f^{p}_s(A) \gtrsim  A^{\frac{d}{p}+s}$ for any $s<0$ and $A\geq 1$ and so 
\[ R \rho_1^{2\sigma} A^{-\frac{d}{p}} N^{-s} f^p_s(A)  \gtrsim \log N
  \gg  (\log N)^{-1}=R\] 
and hence (ii) is satisfied. Thus,  we have $\|\psi(T)\|_{\mathcal{F}L^p_s} \gtrsim  \| U_{2\sigma +1}[\psi_0](T)\|_{\mathcal{F}L^p_s} \gtrsim  \log N.$
Since $\|\psi_0\|_{\mathcal{F}L^p_s} \sim R = (\log N)^{-1}$ and $T\ll N^{-2},$ we get norm inflation by letting $N\to \infty.$
\end{proof}

\section{Norm inflation in modulation spaces}
\label{sec:inflation-modulation}

The proof of Theorem~\ref{theo:inflation-modulation} follows the same
general lines as the proof of Theorem~\ref{theo:inflationFLp} from the
previous section.
Let $N, A$  be dyadic numbers to be specified so that   $N\gg 1$ and $0<A\ll N$. We choose initial data  of the following form
\begin{equation}\label{kidm}
\widehat{\psi_0} =
\begin{cases}
  R A^{-d/2} N^{-s} \chi_{\Omega},\quad\text{if}\quad 1\le q\le 2,\\
  R A^{-d/q} N^{-s} \chi_{\Omega},\quad\text{if}\quad 2\le q\le \infty,
\end{cases}
\end{equation}
where 
\begin{equation*}
  \Omega =
\bigcup_{\eta \in \sum} (\eta + Q_A)  ,\quad Q_A=[-A/2, A/2),
\end{equation*}
for some $\sum \subset \{ \xi \in \R^d:  |\xi| \sim N\}$ such that $\#
\sum \leq 3$. 

\subsection{A priori estimates: $1\le q\le 2$}

Then we have, for any $s\in \R$,
\[ \| \psi_0\|_{H^s} \sim R,  \quad \|\psi_0\|_{M_A} \sim  RN^{-s}.\]

\begin{lemma}\label{dt1}  Let $q\in [1,2]$, $\psi_0$ given by
  \eqref{kidm}, $s<0.$ Then there exists $C>0$ depending  only on $d,
  \sigma$ and $s$ such that following holds.
  \begin{align}
&    \label{dt1p1}  \|U_1[\psi_0](T)\|_{M^{2,q}_s} \leq C R ,\quad
                  \forall T\ge 0,\\
&    \label{dtp2}  \| U_k[\psi_0](T)\|_{M^{2,q}_s} \lesssim
  \rho^{k-1} C^{k} A^{-d/2} RN^{-s} \left\| (1+|n|)^s
  \right\|_{\ell^q\left( 0\leq |n| \leq  A \right)}, 
  \end{align}
where   $\rho=  R N^{-s} A^{d/2} T^{\frac{1}{2\sigma}}$.
\end{lemma}
\begin{proof}
 By Lemma \ref{rl} and   Proposition \ref{uf}, we have
\[\| U_{1}[\psi_0](T)\|_{M^{2,q}_s} \lesssim \|\psi_0(T)\|_{M^{2,q}_s}
  \lesssim  \| \psi_0(T)\|_{M^{2,1}_s} \lesssim R,\]
hence \eqref{dt1p1}.
By  Plancherel theorem and \eqref{fh}, for $s<0,$ we have 
\begin{align*}
\| U_k[\psi_0](T)\|_{M^{2,q}_s} & =  \left\|(1+|n|)^s \|\sigma_n \widehat{U_k[\psi_0]}(T)\|_{L^2} \right\|_{\ell^q}\\
& \leq    \sup_{\xi \in \R^d} \left|  \widehat{U_k[\psi_0]} (t, \xi)\right|  \left\| (1+|n|)^s  \left\| \sigma_n \right\|_{L^2 \left(Q_n \cap\  \operatorname{supp} \widehat{U_k[\psi_0]}(t) \right)} \right\|_{\ell^q}\\
& \leq   \sup_{\xi \in \R^d} \left|  \widehat{U_k[\psi_0]} (t, \xi)\right|  \left\| (1+|n|)^s  \right\|_{\ell^q\left( 0\leq |n| \leq  C A \right)}.
\end{align*}
This yields the desired inequality in \eqref{dtp2}.
\end{proof}

\begin{lemma}\label{dt2} Let $s<0$, $q\in [1,2]$,  $2\leq A \ll N$ and    $\sum=\{ Ne_d, - Ne_d, 2Ne_d\}$ where $e_d= (0,\dots,0, 1) \in \R^d.$ 
If $0<T\ll N^{-2},$ then we   have 
\[ \|U_{2\sigma +1} [\psi_0] (T)\|_{M^{2,q}_s}  \gtrsim   R
  A^{-\frac{d}{2}} N^{-s} \rho^{2\sigma}
  \left\| (1+|n|)^s  \right\|_{\ell^q\left( 0\leq |n| \leq  A \right)},\]
where $\rho=  R N^{-s} A^{d/2} T^{\frac{1}{2\sigma}}.$
\end{lemma}
\begin{proof} By \eqref{fh}, we note that 
\begin{align*}
\|U_{2\sigma +1}[\psi_0](T)\|_{M^{2,q}_s}^q & =   \sum_{n\in \Z^d} \| \square_n(U_{2\sigma +1}[\psi_0](T))\|^q_{L^2} (1+|n|)^{sq} \\
& =  \sum_{n\in \Z^d} \| \sigma_n \widehat{ U_{2\sigma +1}[\psi_0]}(T)\|^q_{L^2} (1+|n|)^{sq}\\
&\gtrsim   \sum_{n\in \Z^d} \frac{1}{(1+|n|)^{-sq}} \left( \int_{Q_n} | \widehat{ U_{2\sigma +1}[\psi_0]}(\xi, T)|^{2} d\xi\right)^{q/2},
\end{align*}
where $Q_n$ is a unit cube centered at $n\in \Z^d.$
Arguing as before in the proof of Lemma~\ref{d2}   (specifically, by \eqref{lu}), for $\xi \in Q_A= [-A/2, A/2)^d$, we  have 
\begin{equation*}
\left| \widehat{U_{2\sigma +1}[\psi_0]}(T, \xi)\right| \gtrsim  \left( R A^{-d/2} N^{-s} \right)^{2\sigma +1}  (A^d)^{2\sigma} T \chi_{(2\sigma +1)^{-1} Q_{A}} (\xi).
\end{equation*}
It follows that 
\begin{align*}
\|U_{2\sigma +1}[\psi_0](T)\|_{M^{2,q}_s}
&\gtrsim     \left( R A^{-d/2} N^{-s} \right)^{2\sigma +1}  (A^d)^{2\sigma} T \left( \sum_{|n|=  \lfloor -A/2 \rfloor}^{\lfloor A/2 \rfloor} \frac{1}{(1+|n|)^{-sq}}  \right)^{1/q}\\
& \gtrsim  R A^{-d/2} N^{-s} \rho^{2\sigma}  \left( \sum_{|n|=  \lfloor -A/2 \rfloor}^{\lfloor A/2 \rfloor} \frac{1}{(1+|n|)^{-sq}}  \right)^{1/q},
\end{align*}
where the floor  function is ${\displaystyle \lfloor x\rfloor =\max\(m\in \Z \mid m\leq x\)}$ and $\rho=  R N^{-s} A^{d/2} T^{\frac{1}{2\sigma}}$.
\end{proof}

\subsection{A priori estimates: $2\le q\le \infty$}

Then we have, for any $s\in \R,$ 
\[ \| \psi_0\|_{\mathcal{F}L^q_s} \sim R,  \quad \|\psi_0\|_{M_A} \sim  R A^{d (\frac{1}{2}-\frac{1}{q})} N^{-s}.\]

\begin{lemma} \label{kb1}  Let $s<0$ and $2 \leq q \leq  \infty.$ Then
  there exists $C>0$ depending  only on $d, \sigma$ and $s$ such that
  following holds.
  \begin{align}
    &\label{skb1}  \|U_1[\psi_0](T)\|_{M^{2,q}_s} \leq C R ,\quad
      \forall T\geq 0,\\
    &\label{skb2}  \| U_k[\psi_0](T)\|_{M^{2,q}_s} \lesssim \rho_2^{k-1} C^{k} A^{-d/q} RN^{-s}  \left\| (1+|n|)^s  \right\|_{\ell^q\left( 0\leq |n| \leq   A \right)},
  \end{align}
where  $\rho_2=  R N^{-s} A^{d \left( 1- \frac{1}{q} \right)} T^{\frac{1}{2\sigma}}.$ 
\end{lemma}
\begin{proof} By Lemma \ref{rl}, we have 
\begin{equation*}
 \|U_1[\psi_0](T)\|_{M_s^{2,q}} \lesssim \|U_1[\psi_0](T)\|_{\mathcal{F}L_s^q}  \leq C R.
\end{equation*}
The proof of \eqref{skb2} is similar to Lemmas \ref{dt1}, \eqref{dtp2}
and \ref{d1}, \eqref{de2}, we omit the details.
\end{proof}
The next lemma is the analogue of Lemmas~\ref{d2} and \ref{dt2}, so we
leave out its proof.
\begin{lemma}\label{kb2}
  Let $s<0$, $2 \leq q \leq  \infty,$  $1\leq A \ll N$ and $\sum=\{ Ne_d, - Ne_d, 2Ne_d\}$ where $e_d= (0,\dots,0, 1) \in \R^d$. If $0<T\ll N^{-2},$ then we   have 
\[ \|U_{2\sigma +1} [\psi_0] (T)\|_{M^{2,q}_s}  \gtrsim  R A^{-\frac{d}{q}} N^{-s} \rho_2^{2\sigma} \left\| (1+|n|)^s  \right\|_{\ell^q\left( 0\leq |n| \leq  A \right)},\]
where  $\rho_2=  R N^{-s} A^{d- \frac{d}{q}} T^{\frac{1}{2\sigma}}$.
\end{lemma}

\subsection{Proof of Theorem \ref{theo:inflation-modulation}}
\begin{lemma}
Let $s<0$.   In the limit $A\to \infty$, we have, for $1\le q<\infty$, 
  \begin{equation*}
    \left\| (1+|n|)^s  \right\|_{\ell^q\left( 0\leq |n| \leq  A
      \right)} \Eq A \infty g_s^q(A):= \begin{cases}  1 \quad \text{if} \  -sq>d,\\ \left( \log  A\right)^{1/q}  \quad \text{if} \ -sq=d,\\
 A^{1/q+s} \quad \text{if} \ -sq<d,
\end{cases}
\end{equation*}
and $\displaystyle \left\| (1+|n|)^s\right\|_{\ell^{\infty}\left( 0\leq |n| \leq  A
    \right)} \Eq A \infty  1$.
\end{lemma}
\begin{proof}
  Since $ (1+|\xi|)^{sq}$ is a decreasing function  in $|\xi|$, in view
  of the integral test and Lemma \ref{dnl}, we have, for  $1\leq q < \infty$,
  \begin{align*}
    \left\| (1+|n|)^s\right\|_{\ell^q\left( 0\leq |n|\leq A \right)}=
  \left( \sum_{0\leq |n|\leq  A} \frac{1}{(1+|n|)^{-sq}} \right)^{1/q}&
  \Eq A \infty \(\int_{|\xi|\le A} \frac{d\xi}{(1+|\xi|)^{-sq}}\)^{1/q} \\
    &
  \Eq A \infty \(\int_0^A \frac{r^{d-1}dr}{(1+r)^{-sq}}\)^{1/q} ,
  \end{align*}
  hence the result for $q$ finite.
The case $q=\infty$ is straightforward.
\end{proof}

To prove Theorem \ref{theo:inflation-modulation}, we distinguish
two cases.
\smallbreak

\noindent {\bf First case:  $1\leq q \leq 2.$}
By Corollary \ref{uc}, we have  the existence of solution  to \eqref{nls} in $M_{A}$  up to time  $T$ whenever  $\rho=RN^{-s} A^{d/2} T^{1/(2\sigma)} \ll 1.$  In view of  Lemma \ref{d1} and since $\rho<1,$  $\sum_{\ell =2}^{\infty} \| U_{2\sigma \ell +1}[\psi_0] (T)\|_{M^{2,q}_s} $ can be dominated by the  sum of a geometric series. Specifically, we have 
\begin{align}\label{mvm1}
\left\| \sum_{\ell =2}^{\infty} U_{2\sigma \ell +1}[\psi_0] (T) \right\|_{M^{2,q}_s} & \lesssim   A^{-d/2} R N^{-s} g_s^q(A) \sum_{\ell=2}^{\infty} \rho^{2\sigma \ell} \nonumber \\
& \lesssim A^{-d/2} R N^{-s} g_s^q(A)\rho^{4\sigma}.
\end{align}
By Corollary \ref{uc} and the triangle  inequality,  we  obtain 
\begin{align*}
\|\psi(T)\|_{M^{2,q}_s} & =  \left\| \sum_{\ell =0}^{\infty} U_{2\sigma \ell +1}[\psi_0] \right\|_{M^{2,q}_s} \\
& \geq  \| U_{2\sigma +1}[\psi_0](T)\|_{M^{2,q}_s} - \|U_1[\psi_0](T)\|_{M^{2, q}_s}- \left\| \sum_{\ell =2}^{\infty} U_{2\sigma \ell +1}[\psi_0](T) \right\|_{M^{2,q}_s}.
\end{align*}
 In order to ensure 
\[ \|\psi(T)\|_{M^{2,q}_s} \gtrsim  \| U_{2\sigma +1}[\psi_0](T)\|_{M^{2,q}_s},   \]
we rely on the conditions
\begin{align}\label{rcm1}
 &\| U_{2\sigma +1}[\psi_0](T)\|_{M^{2,q}_s} \gg
  \|U_1[\psi_0](T)\|_{M^{2,q}_s},\\
  \label{rcm2}
 &\| U_{2\sigma +1}[\psi_0](T)\|_{M^{2,q}_s}  \gg  \left\| \sum_{\ell =2}^{\infty} U_{2\sigma \ell +1}[\psi_0](T) \right\|_{M^{2,q}_s}.
\end{align}
In view  of Lemmas \ref{dt1} and \ref{dt2},   \eqref{rcm1} amount  to the condition
\begin{equation}\label{scm1}
R \rho^{2\sigma} A^{-\frac{d}{2}} N^{-s} g^q_s(A) \gg R
\end{equation}
In view of Lemmas~\ref{dt1} and \ref{dt2}, and \eqref{mvm1}, \eqref{rcm2} amount to the condition 
\begin{equation}\label{scm2}
T\ll N^{-2}, \quad \rho\ll 1, \quad R \rho^{2\sigma} A^{-\frac{d}{2}} N^{-s} g^q_s(A) \gg   R \rho^{4\sigma} A^{-d/2} N^{-s} g_s^q(A).
\end{equation}
We now choose $A, R$ and $T$ so that conditions \eqref{scm1} and \eqref{scm2} are satisfied. To this end,  we set 
\[ R= (\log N)^{-1}, \quad A \sim (\log N)^{-\frac{2\sigma +2}{|s|}} N, \quad  T= (A^{-d/2} N^{s})^{2\sigma}. \]
Then we have   $$\rho= R N^{-s} A^{d/2} T^{\frac{1}{2\sigma}}=(\log N)^{-1}\ll 1.$$ Note that
\[T=  (\log N)^{-\frac{2\sigma +2}{|s|} d (\frac{1}{2} -1) 2\sigma} N^{-d \sigma + 2\sigma s}.\]
Since $s< \frac{d}{2}- \frac{1}{\sigma}$ and $\log N = \mathcal{O}(N^{\epsilon})$ for any $\epsilon>0,$ we   have 
\[ T \ll N^{-2}. \]
We have  $g^{q}_s(A) \gtrsim  A^{\frac{d}{q}+s}$ for any $s<0 $ and   $A\geq 1.$  Thus, for $1\leq q \leq 2,$ we have 
\begin{align*}
R \rho^{2\sigma} A^{-\frac{d}{2}} N^{-s} g^q_s(A)  \gtrsim    (\log
  N)^{- (2\sigma +1)} (\log N)^{2\sigma +2}  A^{d \left( \frac{1}{q}-
  \frac{1}{2} \right)} &\gtrsim \log N \\
  &\gg  (\log N)^{-1}=R,
\end{align*}
 and \eqref{scm1} is satisfied. Thus,  we have $\|\psi(T)\|_{M^{2,q}_s} \gtrsim  \| U_{2\sigma +1}[\psi_0](T)\|_{M^{2,q}_s} \gtrsim  \log N.$
Since $\|\psi_0\|_{M^{2,q}_s} \lesssim R = (\log N)^{-1}$ and $T\ll
N^{-2},$ we get norm inflation by letting $N\to \infty.$ This
completes the proof for $q\in [1,2]$.
\smallbreak

\noindent {\bf Second case: $2\leq q \leq \infty$.}
By Corollary \ref{uc}, we have  the existence of solution  to \eqref{nls} in $M_{A}$  up to time  $T$ whenever  $\rho_2=RN^{-s} A^{d \left( 1- \frac{1}{q}\right)} T^{1/(2\sigma)} \ll 1.$ 
By Lemmas~\ref{kb1} and \ref{kb2},  the conditions
\begin{equation}\label{fdp}
 T\ll N^{-2}, \quad  \rho_2\ll 1 \quad \text{and} \quad  R \rho_2^{2\sigma} A^{-\frac{d}{q}} N^{-s} g^q_s(A) \gg R
\end{equation}
ensures that  
\[ \|\psi(T)\|_{M^{2,q}_s} \gtrsim  \| U_{2\sigma +1}[\psi_0](T)\|_{M^{2,q}_s} \sim  R \rho_1^{2\sigma} A^{-\frac{d}{q}} N^{-s} g^q_s(A).\]
We now choose $A, R$ and $T$ so that conditions  \eqref{fdp} satisfied. To this end,  we set 
\[ R= (\log N)^{-1}, \quad A \sim (\log N)^{-\frac{2\sigma +2}{|s|}} N, \quad  T= (A^{d(\frac{1}{q}-1)} N^{s})^{2\sigma}. \]
Then we have   $$\rho_2= R N^{-s} A^{d- \frac{d}{q}} T^{\frac{1}{2\sigma}}=(\log N)^{-1}\ll 1.$$
Note that
\[T=  (\log N)^{-\frac{2\sigma +2}{|s|} d (\frac{1}{q} -1) 2\sigma} N^{d (\frac{1}{q} -1) 2\sigma + 2\sigma s}.\]
Since $s< d \left( 1-\frac{1}{q} \right)- \frac{1}{\sigma}$ and $\log N = \mathcal{O}(N^{\epsilon})$ for any $\epsilon>0,$ we   have 
\[ T \ll N^{-2}. \]
Note that   $g^{q}_s(A) \gtrsim  A^{\frac{d}{q}+s}$ for any $s<0$ and $A\geq 1$ and so 
\[ R \rho_2^{2\sigma} A^{-\frac{d}{q}} N^{-s} g^q_s(A)  \gtrsim \log N \gg  (\log N)^{-1}=R.\]
 Thus,  we have $\|\psi(T)\|_{M^{2,q}_s} \gtrsim  \| U_{2\sigma +1}[\psi_0](T)\|_{M^{2,q}_s} \gtrsim  \log N.$
Since $\|\psi_0\|_{M^{2,q}_s} \lesssim   R = (\log N)^{-1}$ and $T\ll
N^{-2},$ we get norm inflation by letting $N\to \infty.$ This
completes the proof of Theorem~\ref{theo:inflation-modulation}.

\section{Norm inflation as a by-product of geometric optics}
\label{sec:reduc}
The proof of Theorems~\ref{theo:fourier-lebesgue} and
\ref{theo:modulation} follows the same strategy as in
\cite{carles2012geometric}: through a suitable rescaling, we turn the
ill-posedness result into an asymptotic result, which can be expressed
in the framework of weakly nonlinear geometric optics. More precisely,
we change the unknown function $\psi$ to $u$, via
\begin{equation}
  \label{eq:upsi}
  u^{\eps} (t,x) = \eps^{\frac{2-J}{2\sigma}} \psi (\eps t,  x),
\end{equation}
 where the parameter $\eps$ will tend to zero.
 For   $\psi$ solution to \eqref{nls}, $u^{\eps}$ solves
  \begin{equation}\label{eq:u}
i\eps \partial_t u^{\eps} + \frac{\eps^2}{2}\Delta u^{\eps} = \mu
    \eps^{J}|u^{\eps}|^{2s}u^{\eps}.
\end{equation}
The case $J=1$ corresponds to \emph{weakly nonlinear geometric
  optics} (WNLGO), as defined in \cite{CaBook}. As noticed in
\cite{carles2012geometric} in the framework of Sobolev spaces, a
phenomenon of infinite loss of regularity can be proved via this WNLGO 
setting, under the assumption $s<-1/(2\sigma)$ in the analogue of
Theorems~\ref{theo:fourier-lebesgue} and 
\ref{theo:modulation}. Like in that paper, in order to weaken the
assumption on $s$ to 
$s<-1/(2\sigma+1)$, we will have to consider some value $J>1$, and
perform some ``asymptotic sin'', in the sense that we change the
hierarchy in an asymptotic expansion involving the limit $\eps\to 0$.
\smallbreak

The heuristic idea is the same as in \cite{CCTper}: when negative
regularity is involved, the zero Fourier mode plays a stronger role
than (large) non-zero modes, which come with a small factor. With the
scaling \eqref{eq:upsi} in mind, our goal is to show that we may
consider initial data (for $u$) of the form
\begin{equation*}
  u(0,x) = \sum_{j\not =0}e^{ij\cdot x/\eps}\alpha_j(x),
\end{equation*}
that is containing only rapidly oscillatory terms, and such that the
evolution under \eqref{eq:u} creates a non-trivial non-oscillatory
term.

To be more specific, recall the strategy of multiphase nonlinear geometric
optics (see \cite{carles2010multiphase} for more details): we plug an
\emph{ansatz} of the form 
\begin{equation*}
  u(t,x) = \sum_j e^{i\phi_j(t,x)/\eps} a_j(t,x)
\end{equation*}
into \eqref{eq:u}, and order the powers of $\eps$. The most singular
term is of order $\eps^0$, it is the eikonal equation:
\begin{equation*}
  \d_t \phi_j +\frac{1}{2}|\nabla \phi_j|^2=0.
\end{equation*}
In the case of an initial phase $\phi_j(0,x) = j\cdot x$, no caustic
appears, and the global solution is given by
\begin{equation}\label{eq:phi}
  \phi_j(t,x) = j\cdot x-\frac{|j|^2}{2}t.
\end{equation}
In the sequel, we consider such phases, for $j\in \Z^d$. The next term
in the hierarchy is of order $\eps^1$, but as evoked above, we
``cheat'', and incorporate some nonlinear effects even if $J>1$ (and $J<2$),
\begin{equation}\label{eq:profiles}
  \d_t a_j +j\cdot \nabla_x a_j = -i\mu
  \eps^{J-1}\sum_{\phi_{k_1}-\phi_{k_2}+\dots +\phi_{k_{2\sigma+1}} =
    \phi_j}a_{k_1}\bar a_{k_2}\dots a_{k_{2\sigma+1}}, 
\end{equation}
where we have used $\nabla \phi_j(t,x)=j$. 
Again in view of the specific form of the phase \eqref{eq:phi}, the
condition on the sum involves a \emph{resonant condition},
$(k_1,k_2,\dots,k_{2\sigma+1})\in {\mathcal R}_j$, where
\begin{align*}
 { \mathcal R}_j &= \left\{ (k_\ell)_{1\le \ell\le 2\sigma+1},\
                   \sum_{\ell=1}^{2\sigma+1} (-1)^{\ell+1} k_\ell=j,\
                   \sum_{\ell=1}^{2\sigma+1} (-1)^{\ell+1} |k_\ell|^2=|j|^2\right\}. 
\end{align*}
In the cubic case $\sigma=1$, those sets are described exactly:
\begin{lemma}[See \cite{Iturbulent,carles2010multiphase}]
  Suppose $\sigma=1$. \\
  $\bullet$ If $d=1$, then ${\mathcal R}_j=\{ (j,\ell,\ell),\ (\ell,\ell,j)\ ;\
  \ell\in \Z\}$.\\
  $\bullet$ If $d\ge 2$, then $(k_1,k_2,k_3)\in {\mathcal R}_j$ precisely when
  the endpoints of the vectors $k_1,k_2,k_3,j$ for four corners of a
  non-degenerate rectangle with $k_2$ and $j$ opposing each other, or
  when this quadruplet corresponds to one of the following two
  degenerate cases: $(k_1=j,k_2=k_3)$ or $(k_1=k_2, k_3=j)$. 
\end{lemma}
The above lemma explains why our approach distinguishes the
one-dimensional case and the multi-dimensional case, and in particular
why the cubic one-dimensional case is left out in
Theorems~\ref{theo:fourier-lebesgue} and 
\ref{theo:modulation}. 

\subsection{Multi-dimensional case}

The leading idea in \cite{carles2012geometric} is to start from three
non-trivial modes only, in the case $d\ge 2$, and create at least one
new mode (possibly more if $\sigma\ge 2$), corresponding to $j=0$.
\begin{lemma}\label{lem:creation-zero-multiD}
  Let $d\ge 2$ and $\sigma\in \N^*$. Define $k_1,k_2,k_3\in \Z^d$ as
  \begin{equation*}
    k_1=(1,0,\dots,0), \quad k_2=(1,1,0,\dots,0),\quad
    k_3=(0,1,0,\dots,0). 
  \end{equation*}
  For initial data of the form
  \begin{equation*}
    u(0,x) = \alpha(x)\sum_{j=1}^3 e^{ik_j\cdot x/\eps},\quad \alpha
    \in \Sch(\R^d)\setminus\{0\},
  \end{equation*}
  we have $a_{0\mid t=0}=0$ and $\d_t a_{0\mid
    t=0}=\eps^{J-1}c_0\alpha(x)$, with $c_0=\sharp
  \mathcal R_0\ge 1$. 
\end{lemma}
This lemma is straightforward, in view of \eqref{eq:profiles}, and
since $(k_1,k_2,k_3)\in \mathcal R_0$ if $\sigma=1$,
$(k_1,k_2,k_3,k_1,k_1,\dots,k_1)\in \mathcal R_0$ if $\sigma\ge 2$.

\subsection{One-dimensional case}

In the one-dimensional case, we have a similar result, provided that
the nonlinearity is at least quintic, in view of
\cite[Lemma~4.2]{carles2017norm} (and Example~4.3 there):
\begin{lemma}\label{lem:creation-zero-1Dquintic}
  Let $d=1$ and $\sigma\ge 2$. Define $k_1,k_2,k_3,k_4,k_5\in \Z$ as
  \begin{equation*}
     (k_1,k_2,k_3,k_4,k_5)= (2,-1,-2,4,3).
   \end{equation*}
   For initial data of the form
  \begin{equation*}
    u(0,x) = \alpha(x)\sum_{j=1}^5 e^{ik_j x/\eps},\quad \alpha
    \in \Sch(\R)\setminus\{0\},
  \end{equation*}
  we have $a_{0\mid t=0}=0$ and $\d_t a_{0\mid
    t=0}=\eps^{J-1}c_0\alpha(x)$, with $c_0=\sharp
  \mathcal R_0\ge 1$. 
\end{lemma}

\subsection{How to conclude}

Supposing that we can prove that the geometric expansion recalled
above provides an approximation $u_{\rm app}$ for the solution $u$ to
\eqref{eq:u}, suitable in the sense that the error is measured in a
sufficiently strong norm, the idea is that both $u_{\mid t=0}$ and
$u-u_{\rm app}$ are small in spaces involving negative regularity in $x$,
while Lemma~\ref{lem:creation-zero-multiD} or
\ref{lem:creation-zero-1Dquintic} implies that $u_{\rm app}$ is large
in many spaces.

\section{(Very) weakly nonlinear geometric optics}
\label{sec:wnlgo}

\subsection{A convenient functional framework}

Throughout this section, we denote by $X$ a Banach algebra in the space
variable, that is
\begin{equation}\label{eq:Banach}
\exists C>0,\quad  \|fg\|_X\le C\|f\|_X\|g\|_X,\quad \forall f,g\in X.
\end{equation}
We suppose that $X$ is translation invariant and, denoting by $\tau_kf(x) = f(x-k)$,
  \begin{equation}\label{eq:translation-inv}
    \|\tau_k f \|_X = \|f\|_{X},\quad \forall k\in \R^d,\ \forall f\in X. 
  \end{equation}
We assume in addition that the multiplication by plane wave
oscillations leaves the $X$-norm invariant,
\begin{equation}\label{eq:plane}
  \|f e_k\|_X= \|f\|_X,\quad \forall k\in \R^d,\text{ where } e_k(x)=
  e^{ik\cdot x}. 
\end{equation}
Note that this assumption rules out Sobolev spaces $H^s(\R^d)$, unless
$s=0$. 
Finally, we assume that the Schr\"odinger group acts on $X$, at least
locally in time:
\begin{assumption}\label{hyp:actionX}
  There exists $T_0$ such that $e^{i\frac{t}{2}\Delta}$ maps $X$ to $X$
  for $t\in [0,T_0]$, and
  \begin{equation*}
    \exists C>0,\quad \|e^{i\frac{t}{2}\Delta}\|_{\mathcal L(X,X)}\le C,\quad
    \forall t\in [0,T_0].
  \end{equation*}
\end{assumption}
In \cite{carles2010multiphase,GMS10}, the case $X=\F L^1(M)$ was considered,
with $M=\T^d$ (a choice resumed in \cite{carles2017norm}) or $\R^d$. In
\cite{carles2012geometric,Mo13}, the choice $X= \F L^1\cap L^2(\R^d)$
was motivated 
by the presence of more singular nonlocal nonlinearities. We shall consider later two
sorts of $X$ spaces:  $\F L^1\cap \F L^p$ or $\F  L^1\cap M^{1,1}$. 
\smallbreak

We denote by
\begin{equation*}
  Y=\left\{ (a_j)_{j\in \Z^d},\ \sum_{j\in \Z^d} \|a_j\|_X<\infty\right\}=\ell^1(X),
\end{equation*}
and
\begin{equation*}
  Y_2=\left\{ (a_j)_{j\in \Z^d}\in Y,\ \sum_{j\in \Z^d}
   \( \<j\>^2\|a_j\|_X+\<j\>\|\nabla a_j\|_X+\|\Delta a_j\|_X\)<\infty\right\}.
\end{equation*}
We suppose that $u$ solves \eqref{eq:u} with initial data
\begin{equation*}
  u(0,x) = \sum_{j\in \Z^d}\alpha_j(x) e^{ij\cdot x/\eps}.
\end{equation*}
\subsection{Construction of the approximate solution}

The approximate solution is given by
\begin{equation*}
  u_{\rm app}(t,x) =\sum_{j\in \Z^d} a_j(t,x) e^{i\phi_j(t,x)/\eps},
\end{equation*}
where $\phi_j$ is given by \eqref{eq:phi} and the $a_j$'s solve
\eqref{eq:profiles}, with initial data $\alpha_j$. 
\begin{lemma}\label{lem:exist-profiles}
  Let $d\ge 1$, $\sigma\in \N^*$ and $J\ge 1$. \\
  $\bullet$ If $(\alpha_j)_{j\in \Z^d}\in Y$, then there exists $T>0$
  independent of $\eps\in [0,1]$ and a unique solution $(a_j)_{j\in
    \Z^d}\in C([0,T];Y)$ to the system \eqref{eq:profiles}, such that
  $a_{j\mid t=0}=\alpha_j$ for all $j\in \Z^d$. \\
  $\bullet$ If in addition $(\alpha_j)_{j\in \Z^d}\in Y_2$, then $(a_j)_{j\in
    \Z^d}\in C([0,T];Y_2)$.
\end{lemma}
\begin{proof}
  In view of Duhamel's formula for \eqref{eq:profiles},
  \begin{equation*}
  a_j(t,x) = a_j(0,x-jt) -i\lambda\sum_{(k_1,k_2,\dots,k_{2\sigma+1})\in {\mathcal R}_j} \int_0^t  \(a_{k_1}\bar a_{k_2}\dots a_{k_{2\sigma+1}}\)(s,x - j(t-s))ds,
  \end{equation*}
the first point of the lemma is straightforward, as an easy
consequence of \eqref{eq:translation-inv} and \eqref{eq:Banach},  and a fixed point argument. The second point requires a
little bit more care: it was proven in
\cite[Proposition~5.12]{carles2010multiphase} in the case $X=\F L^1$
(Wiener algebra), and the proof relies only on the properties of $X$
required at the beginning of this section.
\end{proof}
 From now on, we assume
$(\alpha_j)_{j\in \Z^d}\in Y_2$. 
\subsection{Error estimate}
First, in view of the assumptions made in this section, a standard
fixed point argument yields, in view of \eqref{eq:Banach} and
Assumption~\ref{hyp:actionX}: 
\begin{lemma}\label{lem:exists-u}
  Let $d\ge 1$, $\sigma\in \N^*$ and $J\ge 1$. If $u_0\in X$,
  then there exists $T^\eps>0$ and a unique solution $u\in C([0,T^\eps];X)$ to
  \eqref{eq:u} such that $u_{\mid t=0}=u_0$. 
\end{lemma}

To construct the approximate solution $ u_{\rm app}$, we have
discarded two families of terms:
\begin{itemize}
\item Non-resonant terms, involving the source term
 \begin{equation*}
    r_1:= \mu\eps^J\sum_j\sum_{(k_1,k_2,\dots,k_{2\sigma+1})\not\in \mathcal
       R_j}a_{k_1}\bar a_{k_2}\dots a_{k_{2\sigma+1}}
     e^{i(\phi_{k_1}-\phi_{k_2}+\dots +\phi_{2\sigma+1})/\eps}.
  \end{equation*}
\item Higher order terms, involving
  \begin{equation*}
r_2:=    \frac{\eps^2}{2}\sum_j \Delta a_j e^{i\phi_j/\eps}. 
  \end{equation*}
\end{itemize}
Indeed, $u_{\rm app}$ solves
\begin{equation}
  \label{eq:uapp}
  i\eps\d_t u_{\rm app} + \frac{\eps^2}{2}\Delta u_{\rm app}  =\mu
  \eps^J|u_{\rm app} |^{2\sigma}u_{\rm app}  +r_1+r_2.
\end{equation}
Duhamel's formula for $u-u_{\rm app}=:w$ reads
\begin{align*}
  w(t) &= -i\mu \eps^{J-1}\int_0^t
  e^{i\eps\frac{t-\tau}{2}\Delta}\(|u|^{2\sigma}u
         -|u_{\rm app} |^{2\sigma}u_{\rm app} \)(\tau)d\tau\\
       &\quad -i\mu \eps^{J-1}\sum_{j=1,2}\int_0^t
  e^{i\eps\frac{t-\tau}{2}\Delta}r_j(\tau)d\tau.
\end{align*}
In view of our assumptions on $X$, we readily have, thanks to
Minkowski inequality,
\begin{align*}
  \|w(t)\|_X&\le C \int_0^t \( \|u_{\rm app}(\tau)\|_X^{2\sigma} +
              \|w(\tau)\|_X^{2\sigma} \)\|w(\tau)\|_Xd\tau\\
  &\quad + C\eps^{-1}\sum_{j=1,2}\left\|\int_0^t
  e^{i\eps\frac{t-\tau}{2}\Delta}r_j(\tau)d\tau\right\|_X,
\end{align*}
for some $C$ independent of $\eps\in [0,1]$ and $t\in [0,T_0]$. In
view of the second point of Lemma~\ref{lem:exist-profiles}, we readily
have
\begin{equation*}
  \left\|\int_0^t
  e^{i\eps\frac{t-\tau}{2}\Delta}r_2(\tau)d\tau\right\|_X\lesssim
\eps^2. 
\end{equation*}
By construction, $r_1$ is the sum of terms of the
form $g(t,x) e^{ik\cdot x/\eps -\omega t/(2\eps)}$, with $k\in \Z^d$,
$\omega\in \Z$, and the
non-resonance property reads exactly $|k|^2\not =\omega $.
\begin{lemma}\label{lem:non-res-small}
  Let $k\in \R^d$, $\omega\in \R$, with $|k|^2\not =\omega $. Define
  \begin{equation*}
    D^\eps(t,x) =\int_0^t
  e^{i\eps\frac{t-\tau}{2}\Delta} \( g(\tau,x) e^{ik\cdot x/\eps -i\omega
    \tau/(2\eps)}\)d\tau. 
\end{equation*}
Then we have
\begin{align*}
  D^\eps(t,x) &= \frac{-2i\eps}{|k|^2-\omega}
  e^{i\eps\frac{t-\tau}{2}\Delta} \(g(\tau,x)e^{ik\cdot
                x/\eps-i\omega
    \tau/(2\eps)}\)\Big|_0^t\\
              +\frac{2i\eps}{|k|^2-\omega}&
                \int_0^t  e^{i\eps\frac{t-\tau}{2}\Delta} \(
                e^{ik\cdot x/\eps-i\omega \tau/(2\eps)}\(\frac{i}{2}\(\eps\Delta g +2k\cdot
                \nabla g\)+\d_t g\)(\tau,x)\)d\tau.  
\end{align*}
In particular, for $t\in [0,T_0]$,
\begin{align*}
  \|D^\eps(t)\|_X\lesssim \frac{\eps}{\left\lvert |k|^2-\omega
    \right\rvert} &\big(\|g\|_{L^\infty([0,t];X)} +
  \|\Delta g\|_{L^\infty([0,t];X)} + |k| \|\nabla  g\|_{L^\infty([0,t];X)}\\
                  &\quad +\|\d_t g\|_{L^\infty([0,t];X)}\big).
\end{align*}
\end{lemma}
\begin{proof}
  The last estimate follows directly from the identity of the lemma,
  \eqref{eq:plane} and Assumption~\ref{hyp:actionX}, so we only address the
  identity, which is essentially established in
  \cite[Lemma~5.7]{carles2010multiphase} (up to the typos
  there). Setting $\eta = \xi 
  -k/\eps$, the (spatial) Fourier 
  transform of $D$ is given by 
 \begin{align*}
    \widehat D^\eps(t,\xi)&= e^{-i\eps t|\eta+k/\eps|^2/2}\int_0^t
    e^{i\eps\tau|\eta+k/\eps|^2/2} \,
    \hat b\(\tau,\eta\)e^{-i\omega
      \tau/(2\eps)}d\tau\\
&=e^{-i\eps t|\eta+k/\eps|^2/2}\int_0^t
    e^{i\tau\theta/2} \,
                             \hat b\(\tau,\eta\)d\tau \\
   &=e^{-i\eps t|\eta+k/\eps|^2/2}\int_0^t
    e^{i\tau\theta_2/2} \,e^{i\tau\theta_1/2}\, 
    \hat b\(\tau,\eta\)d\tau ,
  \end{align*}
where we have denoted 
\begin{equation*}
  \theta = \eps \left\lvert \eta+\frac{k}{\eps}\right\rvert^2
  -\frac{\omega}{\eps} = \underbrace{\eps|\eta|^2 +2 k\cdot
    \eta}_{\theta_1} +\underbrace{\frac{|k|^2-\omega}{\eps}}_{\theta_2}. 
\end{equation*}
Integrate by parts, by first integrating $e^{i\tau\theta_2/2}$: 
\begin{equation*}
 e^{i\eps\frac{t}{2}|\xi|^2} \widehat D^\eps(t,\xi) = -\frac{2i}{\theta_2} e^{i\tau\theta/2} 
    \hat b\(\tau,\eta\)\Big|_0^t +\frac{2i}{\theta_2}\int_0^t
    e^{i\tau\theta/2} \( i\frac{\theta_1}{2}\widehat
    b\(\tau,\eta\)+
\widehat {\d_t b}\(\tau,\eta\)\)d\tau.
\end{equation*}
The identity follows by inverting the Fourier transform. 
\end{proof}
We infer:
\begin{proposition}\label{prop:errorWNLGO}
  Let $d\ge 1$, $\sigma\in \N^*$, $J\ge 1$, and $(\alpha_j)_{j\in
    \Z^d}\in Y_2$. Then for $T$ as in Lemma~\ref{lem:exist-profiles},
  \begin{equation*}
    \|u-u_{\rm app}\|_{L^\infty([0,T];X)}\lesssim \eps. 
  \end{equation*}
\end{proposition}
\begin{proof}
  First, Lemma~\ref{lem:exist-profiles} and \eqref{eq:profiles} imply
  that we also have $(\d_t a_j)_{j\in \Z^d}\in C([0,T];Y)$. Then, in
  view of these properties and Lemma~\ref{lem:non-res-small}, we have
  \begin{equation*}
  \left\|\int_0^t
  e^{i\eps\frac{t-\tau}{2}\Delta}r_1(\tau)d\tau\right\|_X\lesssim
\eps^{J+1},
\end{equation*}
where we have used the fact that in the application of
Lemma~\ref{lem:non-res-small}, $\left\lvert
  |k|^2-\omega\right\rvert\ge 1$, since now $k\in \Z^d$ and $\omega\in
\Z$. 
We infer
\begin{equation*}
  \|w(t)\|_X\le C \int_0^t \( \|u_{\rm app}(\tau)\|_X^{2\sigma} +
              \|w(\tau)\|_X^{2\sigma} \)\|w(\tau)\|_Xd\tau + C\eps^J + C\eps,
            \end{equation*}
where $C$ is independent of $\eps\in [0,1]$ and $t\in
[0,T]$. Lemmas~\ref{lem:exist-profiles} and 
\ref{lem:exists-u}  yield $w\in C([0,\min(T,T^\eps)];X)$. Since
$w_{\mid t=0}=0$, the
above inequality and a standard
continuity argument imply that $u\in     C([0,T];X)$     provided
that $\eps>0$ is sufficiently small, along with the announced error
estimate. 
\end{proof}

\section{Norm inflation with infinite loss of regularity}
\label{sec:conclusion}

\subsection{Proof of Theorem~\ref{theo:fourier-lebesgue}}

For $1<J<2$ to be fixed later, let $u^\eps$ defined by
\eqref{eq:upsi}, and consider the initial data given by
Lemma~\ref{lem:creation-zero-multiD} (if $d\ge 2$) or
Lemma~\ref{lem:creation-zero-1Dquintic} (if $d=1$ and $\sigma\ge 2$). 
We apply the analysis from Section~\ref{sec:wnlgo} with 
$X=\F L^1\cap \F L^\infty$. This is obviously a Banach algebra,
\eqref{eq:Banach} holds, thanks
to Young inequality, the $X$-norm is invariant by translation, and by
multiplication by 
plane wave oscillations as in \eqref{eq:plane}.  Assumption~\ref{hyp:actionX} is satisfied with $C=1$
for any $T_0>0$, since the Schr\"odinger group is a Fourier multiplier
of modulus one. We can therefore invoke the conclusions of Lemma~\ref{lem:creation-zero-multiD} (if $d\ge 2$),
Lemma~\ref{lem:creation-zero-1Dquintic} (if $d=1$ and $\sigma\ge 2$),
and Proposition~\ref{prop:errorWNLGO} (in all cases). In order to
translate these properties involving $u$ solving \eqref{eq:u} in terms
of $\psi$ solving \eqref{nls}, we use the following lemma:
\begin{lemma}\label{lem:scaling-fourier} Let $d\geq 1$.
  For $f\in \mathcal{S}'(\R^d)$ and $j \in \R^d,$ denote 
\[ I^{\eps}(f, j)(x) =  f(x) e^{ij \cdot x/\eps}.\]
\begin{itemize}
\item  For all $s \in \mathbb R$, $p\in [1,\infty]$, and $f\in
  \mathcal{F}L^{p}_{s}(\R^d)$, $\|I^{\eps}(f,
  0)\|_{\mathcal{F}L^{p}_{s}} =
  \|f\|_{\mathcal{F}L^p_{s}}$.
\item  Let $j \in \R^d\setminus\{ 0\}$. For all
  $s \leq 0$, there exists $C= C(j)$ independent of $p\in [1,\infty]$ such that  for
  all $f\in \mathcal{S}(\R^d)$,
\[ \| I^{\eps} (f, j) \|_{\mathcal{F}L^{p}_{s}}\le C
    \eps^{  |s| } 
    \|f\|_{\mathcal{F}L^{p}_{|s|}}. \] 
\end{itemize}  
 \end{lemma}
 \begin{proof}
   We have obviously
 \[ \widehat{ I^{\eps}(f, j)} (\xi)=  \hat{f} \left( \xi-  \frac{ j}{\eps}\right).\]
The first point is thus  trivial. For the second one,
if $p$ is finite,
\begin{equation*}
\| I^{\eps}(f, j)\|_{ \mathcal{F}L^p_{s}}^p =\int \langle \xi \rangle ^{ps}  \left|  \hat{f} \left( \xi-  \frac{ j}{\eps}\right) \right|^p d\xi.
\end{equation*} 
Note that, for $s \leq 0$,
 \begin{align*}
 \| I^{\eps}(f, j)\|_{ \mathcal{F}L^p_{s}}^p& = \int
  \<\xi\>^{ps}\<\xi-\frac{j}{\eps}\>^{ps} \<\xi-\frac{j}{\eps}\>^{p|s|}
     \left|  \hat{f} \left( \xi-  \frac{ j}{\eps}\right) \right|^p d\xi\\
   & \le  \sup_{\xi \in \R^d} \(\<\xi\>^{-1}\<\xi-\frac{j}{\eps}\>^{-1}
     \)^{p|s|} \|f\| _{ \mathcal{F}L^p_{|s|}}^p                                          
 \end{align*}
 For $j\not =0$,
 \begin{equation*}
 \inf_{\xi\in \R^d}  \<\xi\>\<\xi-\frac{j}{\eps}\>\gtrsim \frac{1}{\eps},
 \end{equation*}
 hence the second point of the lemma in the case $p$ finite. The case
 $p=\infty$ follows from the same estimate, controlling the supremum of
 a product by the product of the suprema. 
 \end{proof}
With $u_{\mid t=0}$ as in Lemma~\ref{lem:creation-zero-multiD}  or
Lemma~\ref{lem:creation-zero-1Dquintic}, the above lemma yields, in
view of \eqref{eq:upsi}, and for $s<0$,
\begin{equation*}
  \|\psi(0)\|_{\F L^p_s}\lesssim \eps^{\frac{J-2}{2\si}+|s|}. 
\end{equation*}
This sequence of initial data is small (in $\F
 L^p_s$ for all $p$) provided that
 \begin{equation}\label{eq:cond-s}
   |s|>\frac{2-J}{2\si}.
 \end{equation}
 Lemma~\ref{lem:creation-zero-multiD}  and
Lemma~\ref{lem:creation-zero-1Dquintic} show that there exists
$\tau>0$ independent of $\eps$ such that
\begin{equation*}
  \|a_0(\tau)\|_{\F L^p_k}\gtrsim \eps^{J-1}, \quad \forall p\in
  [1,\infty],\ \forall k\in \R.
\end{equation*}
By construction,
\begin{equation*}
  \hat u_{\rm app}(t,\xi) = \sum_{j\in \Z^d}e^{-it\frac{|j|^2}{2\eps}}\hat
    a_j\(t,\xi-\frac{j}{\eps}\),
  \end{equation*}
  so we infer, at least for $\eps$ sufficiently small ($(a_j(\tau))_{j\in \Z^d}\in  Y$ from Lemma~\ref{lem:exist-profiles}),
 \begin{equation*}
  \|u_{\rm app}(\tau)\|_{\F L^p_k}\gtrsim \eps^{J-1}, \quad \forall p\in
  [1,\infty],\ \forall k\in \R.
\end{equation*}
On the other hand, since $X=\F L^1\cap \F L^\infty\subset \F L^p$,
Proposition~\ref{prop:errorWNLGO} yields
\begin{equation*}
  \|u(\tau)-u_{\rm app}(\tau)\|_{\F L^p_k} \lesssim \|u(\tau)-u_{\rm
    app}(\tau)\|_{X}\lesssim \eps,\quad \forall k\le 0, 
\end{equation*}
hence, if $J<2$,
\begin{equation*}
   \|u(\tau)\|_{\F L^p_k}\gtrsim \eps^{J-1}, \quad \forall p\in
  [1,\infty],\ \forall k\le 0.
\end{equation*}
Therefore,
\begin{equation*}
  \| \psi(\eps\tau)\|_{\F L^p_k}\gtrsim
    \eps^{\frac{J-2}{2\si}}\times \eps^{J-1}, \quad \forall p\in
  [1,\infty],\ \forall k\in \R.
  \end{equation*}
  The right hand side is unbounded as $\eps\to 0$ provided that
  \begin{equation*}
    \frac{J-2}{2\si}+J-1<0,\quad \text{that is},\quad J<\frac{2\si+2}{2\si+1}.
  \end{equation*}
  Then given $s<-1/(2\si+1)$, we can always find a $J\in ]1,2[$
  satisfying \eqref{eq:cond-s} and the above constraint.
  Theorem~\ref{theo:fourier-lebesgue} follows in
  the case $k\le 0$, by taking for instance $\eps_n=1/n$ and $t_n=
  \eps_n\tau$. In the case $k>0$, we just recall the obvious estimate
  \begin{equation*}
    \| \psi(\eps\tau)\|_{\F L^p_k}\ge \| \psi(\eps\tau)\|_{\F L^p},
      \end{equation*}
      and the proof of Theorem~\ref{theo:fourier-lebesgue} is
      complete. 
\subsection{Proof of Theorem~\ref{theo:modulation}}

In the case of modulation spaces, the proof goes along the same lines
as above, up to adapting the space $X$ and
Lemma~\ref{lem:scaling-fourier}. 
\smallbreak

We choose $X=\F L^1\cap M^{1,1}$. Theorem~\ref{theo:algebra} shows that the
Banach algebra property \eqref{eq:Banach} is satisfied. $\F L^1$ is
translation invariant, and for $M^{1,1}$,
\begin{align*}
  V_g\(\tau_k f\)(x,y)&= \int_{\R^d} f(t-k)
  \overline{g(t-x)}e^{-iy\cdot t}dt = e^{-iy\cdot k}\int_{\R^d} f(t)
                        \overline{g(t+k-x)}e^{-iy\cdot t}dt\\
  &=e^{-iy\cdot k} V_g(f)(x-k,y),
\end{align*}
and thus
\begin{equation*}
   \|\tau_k f \|_{M^{1,1}} = \|V_g\(\tau_k f\)\|_{L^1_{x,y}}  = \|f\|_{M^{1,1}} . 
\end{equation*}
We have used
already the fact that \eqref{eq:plane} is satisfied on $\F L^1$. On
$M^{1,1}$, this is the case too, since for $k\in \R^d$,
\begin{equation*}
  V_g\(f e_k\) (x,y) = \int_{\R^d} f(t) e^{ik\cdot t}
  \overline{g(t-x)}e^{-i y\cdot t}dt =  V_g\(f \) (x,y-k),
\end{equation*}
and so
\begin{align*}
  \|fe_k\|_{M^{1,1}} = \|V_g\(fe_k\)\|_{L^1_{x,y}} =
  \|V_g\(f\)\|_{L^1_{x,y}} = \|f\|_{M^{1,1}} . 
\end{align*}
Finally, Assumption~\ref{hyp:actionX} is satisfied thanks to
Proposition~\ref{uf}, and we can again invoke
Lemma~\ref{lem:creation-zero-multiD}, 
Lemma~\ref{lem:creation-zero-1Dquintic}, 
and Proposition~\ref{prop:errorWNLGO}.
\smallbreak

Like before, Lemma~\ref{lem:creation-zero-multiD} and 
Lemma~\ref{lem:creation-zero-1Dquintic} show that there exists
$\tau>0$ independent of $\eps$ such that
\begin{equation*}
  \|a_0(\tau)\|_{M^{p,q}_k}\gtrsim \eps^{J-1}, \quad \forall p,q\in
  [1,\infty],\ \forall k\in \R.
\end{equation*}
By the same asymptotic decoupling phenomenon as in the case of $\F
L^p$ spaces, we infer
\begin{equation*}
  \|u_{\rm app}(\tau)\|_{M^{p,q}_k}\gtrsim \eps^{J-1}, \quad \forall p,q\in
  [1,\infty],\ \forall k\in \R.
\end{equation*}
In view of Lemma~\ref{rl}, $X
\hookrightarrow M^{p,q}_k$ for all $p,q\ge 1$ and all $k\le 0$, and so
\begin{equation*}
  \|u(\tau)-u_{\rm app}(\tau)\|_{M^{p,q}_k}\lesssim  \|u(\tau)-u_{\rm
    app}(\tau)\|_X\lesssim \eps, \quad \forall p,q\in
  [1,\infty],\ \forall k\le 0. 
\end{equation*}
The analogue of Lemma~\ref{lem:scaling-fourier} is the following:
\begin{lemma}\label{lem:scaling-modulation} Let $d\geq 1$. 
  For $f\in \mathcal{S}'(\R^d)$ and $j \in \R^d,$ denote 
\[ I^{\eps}(f, j)(x) =  f(x) e^{ij \cdot x/\eps}.\]
\begin{itemize}
\item  For all $s \in \mathbb R$, $p,q\in [1,\infty]$, and $f\in
  M^{p,q}_s(\R^d)$, $\|I^{\eps}(f,
  0)\|_{M^{p,q}_s} =
  \|f\|_{M^{p,q}_s}$.
\item  Let $j \in \R^d\setminus\{ 0\}$. For all
  $s \leq 0$, there exists $C= C(j)$ independent of $p,q\in [1,\infty]$ such that  for
  all $f\in \mathcal{S}(\R^d)$,
\[ \| I^{\eps} (f, j) \|_{M^{p,q}_s}\le C
    \eps^{  |s| } 
    \|f\|_{M^{p,q}_s}. \] 
\end{itemize}  
 \end{lemma}
 \begin{proof}
   The first point is proven like \eqref{eq:plane} above. For the
   second point, write
\begin{align*}
\| I^{\eps} (f, j) \|_ {M^{p,q}_s} & = \left\| \left\|V_{g}f\(x,y-\frac{j}{\eps}\)\right\|_{L^p_x} \<y\>^s\right\|_{L_y^q}
  =  \left\| \|V_{g}f(x,y)\|_{L^p_x} \<y+ \frac{j}{ \eps}\>^s\right\|_{L_y^q}\\
& \le\<\frac{j}{\eps}\>^s   \left\| \|V_{g}f(x,y)\|_{L^p_x}
   \<y\>^{|s|}\right\|_{L_y^q} ,
\end{align*}
where we have used Peetre inequality $\<a+b\>^s\le
\<a\>^s\<b\>^{|s|}$. The lemma follows.
 \end{proof}
At this stage, we can repeat the same arguments as in the previous
subsection, and Theorem~\ref{theo:modulation} follows.

\end{document}